\documentclass[reqno,11pt]{amsart}
\usepackage{amsmath,amssymb,latexsym,soul,cite,mathrsfs,ragged2e}
\usepackage{color,enumitem,graphicx}
\usepackage[colorlinks=true,urlcolor=blue,
citecolor=red,linkcolor=blue,linktocpage,pdfpagelabels,
bookmarksnumbered,bookmarksopen]{hyperref}
\usepackage[english]{babel}
\usepackage[left=2.6cm,right=2.6cm,top=2.9cm,bottom=2.9cm]{geometry}
\usepackage{tikz}
\usepackage{pgfplots}
\usetikzlibrary{patterns, positioning}
\usepackage{mathtools}
\usepackage{xcolor}
\usepackage{comment}
\definecolor{darkgreen}{RGB}{0,100,0}
\newtheorem{theorem}{Theorem}
\newtheorem{lemma}[theorem]{Lemma}

\newtheorem{proposition}[theorem]{Proposition}
\newtheorem{remark}[theorem]{Remark}
\newtheorem{definition}[theorem]{Definition}

\newtheorem{theoremletter}{Theorem}

\newtheoremstyle{tttheorem}
{}                
{}                
{\slshape}        
{}                
{\bfseries}       
{'}               
{ }               
{}                
\theoremstyle{tttheorem}

\def\XXint#1#2#3{{\setbox0=\hbox{$#1{#2#3}{\int}$ }
		\vcenter{\hbox{$#2#3$ }}\kern-.6\wd0}}

\newenvironment{boldromanenum}
{\begin{enumerate}[label=\textbf{(\roman*)}, font=\normalfont]}
{\end{enumerate}}

\title[Minimax solutions for FitzHugh-Nagumo systems]{Minimax solutions for FitzHugh-Nagumo systems with critical exponential Growth}
\author[J.M. do \'O]{Jo\~{a}o Marcos do \'O}
\author[J. C. de Albuquerque]{Jos\'{e} Carlos de Albuquerque}
\author[H. H. C. e Silva]{Hugo H. C. e Silva}

\address[J.M. do \'O]{Department of Mathematics,
	Federal University of Para\'{\i}ba
	\newline\indent 
	58051-900, Jo\~ao Pessoa-PB, Brazil \newline\indent Orcid: 0000-0001-7039-4365}
\email{\href{mailto:jmbo@mat.ufpb.br}{jmbo@mat.ufpb.br}}

\address[J. C. de Albuquerque]{Department of Mathematics,
	Federal University of Pernambuco
	\newline\indent 
	50670-420, Recife-PE, Brazil \newline\indent Orcid: 0000-0003-2273-6054}
\email{\href{mailto: josecarlos.melojunior@ufpe.br}{josecarlos.melojunior@ufpe.br}}

\address[H. H. C. e Silva]{
    Department of Mathematics,
	Federal University of Para\'{\i}ba
	\newline\indent 
	58051-900, Jo\~ao Pessoa-PB, Brazil \newline\indent Orcid: 0009-0007-2312-032X}
\email{\href{mailto:hugo.henryque@ufpe.br}{hugo.henryque@ufpe.br}}

\pgfplotsset{compat=1.18}
\begin{document}

\begin{abstract}
We study a class of FitzHugh-Nagumo type systems with critical exponential growth in a smooth bounded domain of $\mathbb{R}^{2}$. The classical reduction method turns the strongly indefinite system into a scalar nonlocal problem whose quadratic part induces a norm equivalent to, but different from the standard Sobolev norm, so that the Trudinger-Moser threshold must be handled in this framework. We prove the existence of a ground state solution via a mountain pass argument
combined with a Nehari level comparison. We also prove the existence, without imposing the monotonicity condition
required in the ground state argument, of a nontrivial solution obtained
through a saddle-point linking argument based on the spectral decomposition
of $\operatorname{T}=-\Delta+B$. \\

\noindent \textbf{Keywords:} FitzHugh-Nagumo systems; critical exponential growth; Trudinger-Moser inequality; ground state solutions; saddle-point geometry; linking theorem; nonlocal problem; Moser functions.\\
\textbf{MSC (2020):} 35J20; 35J47; 35J50; 35B33; 35Q92.
\end{abstract}

	\maketitle
	
	
	\begin{center}
		\footnotesize
	\end{center}

%

\vspace{-1.0cm}



\section{Introduction}

The FitzHugh-Nagumo system is a classical example of a reaction-diffusion system arising in the mathematical description of excitable media. Its origin goes
back to the classical work of Hodgkin and Huxley \cite{Hodgkin1952Quantitative},
where a quantitative model for the conduction and excitation of nerve impulses was
proposed. Although the Hodgkin-Huxley model is highly successful from the
biophysical point of view, its mathematical structure is rather involved. In order
to retain the essential activator-inhibitor mechanism while obtaining a more
tractable model, FitzHugh \cite{FitzHugh1961Impulse} and, independently, Nagumo,
Arimoto and Yoshizawa \cite{Nagumo1962Active} introduced a simplified system for
the propagation of electrical impulses along nerve axons. In these models, the
first component represents an activator variable, while the second one describes a
recovery or inhibitor variable.

From this point of view, FitzHugh-Nagumo equations may be regarded as a
prototype of reaction-diffusion systems of the form
\begin{equation*}
\begin{cases}
u_t = D_1 \Delta u + g(u) - v, & \text{in } (0,+\infty) \times \Omega,\\
v_t = D_2 \Delta v + \varepsilon(u-\gamma v), & \text{in } (0,+\infty) \times \Omega,
\end{cases}
\end{equation*}
where $\Omega \subset \mathbb{R}^{N}$ is a domain and $D_1,D_2,\varepsilon,\gamma$
are positive parameters. Stationary states, standing waves and related elliptic
systems associated with this model have been widely investigated in several
directions, including bifurcation, stability, concentration phenomena, peak
solutions, multiplicity and variational methods. See, for instance,
\cite{MR0819214,Conley1986Bifurcation,Figueiredo1986Maximum,Dancer2004Minimization,
Dancer2005Multipeak,Dancer2006Interior,Ren2005Nucleation,Wei2005Clustered,
Wei2006Standing,Sweers2003Bifurcation,Chen2014Variational,Alvesfig} and the
references therein.

In the present paper, we study the existence of a ground state solution
and of a nontrivial solution obtained through a saddle-point linking
argument for the following class of FitzHugh-Nagumo type systems
\begin{equation} \label{SC}\tag{$S$}
\begin{cases}
-\Delta u = -\delta v + f(x,u), & \text{in } \Omega, \\
-\Delta v = \delta u - \gamma v, & \text{in } \Omega, \\
u = v = 0, & \text{on } \partial \Omega,
\end{cases}
\end{equation}
where $\Omega$ is a smooth bounded domain in $\mathbb{R}^{2}$, $\delta,\gamma > 0$ and $f$ has critical exponential growth in the Trudinger-Moser sense.

A first difficulty in studying system \eqref{SC} is that the natural associated energy
functional is strongly indefinite. Its quadratic part
has opposite signs in the variables $u$ and $v$, and therefore neither direct
minimization nor the standard mountain pass argument can be applied to the
system in a straightforward way. A method for overcoming this difficulty was introduced by de Figueiredo
and Mitidieri \cite{Figueiredo1986Maximum} in the analysis of elliptic systems of
FitzHugh-Nagumo type. The idea is to solve the second equation of the system with
respect to $v$ and substitute this expression into the first equation. In this
way, the original system is reduced to a scalar elliptic equation with a nonlocal
term. More precisely, since $\gamma>0$, the bilinear form
\[
\operatorname{\sigma}(v,\varphi)=\int_{\Omega}\nabla v\nabla\varphi\, \mathrm{d}x
+\gamma\int_{\Omega}v\varphi\, \mathrm{d}x
\]
is coercive on $H_0^1(\Omega)$. Hence, by the Lax-Milgram theorem, the operator
$(-\Delta+\gamma)$, endowed with homogeneous Dirichlet boundary conditions, is invertible in the weak sense. Thus the second equation gives
\begin{equation*}
v=\delta(-\Delta + \gamma)^{-1}u.    
\end{equation*}
Consequently,
\begin{equation*}
\delta v=\delta^2(-\Delta + \gamma)^{-1}u \coloneqq B[u],    
\end{equation*}
and the system \eqref{SC} can be rewritten as the nonlocal problem
\begin{equation}\label{nonlocalproblem} \tag{$P$}
\begin{cases}
-\Delta u+B[u]=f(x,u), & \text{in } \Omega,\\
u=0, & \text{on } \partial\Omega,
\end{cases}
\end{equation}
where $B$ is the linear resolvent operator associated with the second equation
of \eqref{SC}. This reduction removes the strong indefiniteness of the original system, but it
introduces a different difficulty. The reduced equation is no longer a standard
local scalar equation: the operator $B$ produces a global contribution in the
energy and induces a natural norm on $H^{1}_{0}(\Omega)$ which is equivalent to
the usual Sobolev norm, but not identical to it. This distinction is essential in the critical exponential regime. Although the sharp Trudinger-Moser constant remains \(4\pi\), the normalization of
concentrating sequences and the corresponding minimax estimates must be
adapted to the nonlocal norm generated by the FitzHugh-Nagumo coupling.
We shall also use the spectral structure of the operator
\begin{equation*}
\operatorname{T} \coloneqq -\Delta+B: L^{2}(\Omega) \rightarrow L^{2}(\Omega),
\end{equation*}
with $D(T) = H^{2}(\Omega) \cap H^{1}_{0}(\Omega).$ This operator was introduced and studied by de Figueiredo and Mitidieri
\cite{Figueiredo1986Maximum} in their analysis of FitzHugh-Nagumo type
elliptic systems, where its spectral structure was established in detail. We
recall here the facts that will be used throughout the linking argument, all
of which are due to \cite{Figueiredo1986Maximum}. If
$(\lambda_j)_{j \in \mathbb{N}}$ denotes the sequence of Dirichlet eigenvalues
of $-\Delta$ in $\Omega$, counted with multiplicity, and
$(\varphi_j)_{j \in \mathbb{N}}$ is a corresponding $L^2$-orthonormal basis of
eigenfunctions, then $B[\varphi_j] = \dfrac{\delta^2}{\lambda_j + \gamma}\varphi_j$,
so that the same eigenfunctions $\varphi_j$ diagonalize $\operatorname{T} = -\Delta + B$, with
eigenvalues
\begin{equation} \label{eq:eigenvalues_intro}
\widehat{\lambda}_{j} = \lambda_{j} + \frac{\delta^2}{\lambda_{j} + \gamma}, \qquad j \in \mathbb{N};
\end{equation}
see \cite[eq.~$(1.2)$]{Figueiredo1986Maximum}. Moreover, it was shown in
\cite[Proposition~1.4]{Figueiredo1986Maximum} that these are the only
eigenvalues of $\operatorname{T}$ and that its spectrum consists precisely of the sequence
$(\widehat{\lambda}_j)_{j \in \mathbb{N}}$. In particular, although $\operatorname{T}$ is
nonlocal, its eigenspaces are generated by the Dirichlet eigenfunctions of the
Laplacian, so that the finite-dimensional space used in the linking construction
may be defined solely in terms of $\varphi_{1}, \dots, \varphi_{k}$. Finally, under
the assumption $\delta \leq \lambda_{1} + \gamma$, the sequence
$(\widehat{\lambda}_j)_{j \in \mathbb{N}}$ is nondecreasing, since the map
$t \mapsto t + \delta^2/(t+\gamma)$ is nondecreasing for $t \geq \lambda_1$. The
monotonicity of $(\widehat{\lambda}_j)_{j\in\mathbb{N}}$ and the sharp condition
ensuring it were already discussed in \cite[Remark~1.2]{Figueiredo1986Maximum}. This spectral decomposition will be used in the saddle point part of the paper to
construct a linking geometry.

Another feature of the system \eqref{SC} is that $f$ has critical exponential growth motivated by the Trudinger-Moser inequality. In its sharp form, due to Moser \cite{moser1971sharp}:
\begin{theoremletter}(Trudinger-Moser) \label{eq:TMclassicaC}There exists a constant $C > 0$ such that
\begin{equation*}
    \sup \limits_{ \{u \in H^{1}_{0}(\Omega) \;  \vert  \; \Vert \nabla u \Vert_{L^{2}(\Omega)} \leq 1 \}} \int_{\Omega} \exp(\alpha u^{2}) \; \mathrm{d}x \leq C \vert \Omega \vert,  \quad \text{for every } \; 0 < \alpha \leq 4 \pi.
\end{equation*}
Moreover, the constant $4\pi$ is sharp, in the sense that the above supremum is infinite for every $\alpha>4\pi$.
\end{theoremletter}
Problems involving Trudinger-Moser critical growth have been studied in different
elliptic settings. Adimurthi \cite{adimurthi1990existence} considered Dirichlet
problems in the framework of the $N$-Laplacian. For semilinear elliptic equations
in $\mathbb{R}^{2}$, de Figueiredo, Miyagaki and Ruf \cite{de2014elliptic} proved
a compactness result based on the condition that the minimax level lies below the
critical energy level $2\pi/\alpha_{0}$, the threshold associated with the
Trudinger-Moser inequality in the variational setting. In a related direction, Barboza and do \'{O} \cite{10.57262/die/1580958028} studied H\'enon elliptic equations in
$\mathbb{R}^{2}$ with subcritical and critical exponential growth. Figueiredo and Severo
\cite{figueiredo2016ground} considered a Kirchhoff problem with exponential
critical growth, where a nonlocal term is present. To the best of our knowledge,
only a few works deal with FitzHugh-Nagumo type systems involving critical
growth phenomena. Figueiredo and Montenegro \cite{FigueiredoMontenegro2021}
studied a zero mass FitzHugh-Nagumo system with critical growth, treating both
polynomial and exponential regimes. Torres Ledesma \cite{TorresLedesma2022}
considered fractional FitzHugh-Nagumo type systems in $\mathbb{R}^{N}$, allowing
critical polynomial growth when $N\geq 2$ and critical exponential growth in the
case $N=1$ and $s=1/2$. More recently, Carvalho, Furtado and Melo
\cite{carvalho2025fitzhugh} studied a two-dimensional nonlinear FitzHugh-Nagumo system with exponential growth and radial weights in $\mathbb{R}^{2}$. These works indicate that FitzHugh-Nagumo type systems with critical growth
remain relatively unexplored, especially in bounded two-dimensional domains
where the reduction procedure leads to a nonlocal norm interacting with the
Trudinger-Moser threshold. 

In contrast with the preceding works, we consider a smooth bounded domain in dimension two and exploit the resolvent reduction of the strongly indefinite system. The resulting quadratic form induces a nonstandard norm, requiring both a sharp Trudinger–Moser inequality in this norm and a projected Moser construction compatible with the spectral splitting used in the linking argument. 

In this work, we assume that $f\in C(\overline{\Omega}\times\mathbb{R},\mathbb{R})$ and that the following
conditions hold:

\begin{enumerate}[
label=$(f_{\arabic*})$,
ref=$(f_{\arabic*})$,
series=f
]
\item \label{f1}
There exists $\alpha_0>0$ such that $f$ has critical exponential growth at infinity, namely 
\begin{equation*}
\lim_{\vert s\vert\rightarrow + \infty} \vert f(x,s) \vert \exp(-\alpha s^{2})
=
\begin{cases}
0, & \text{for all } \alpha>\alpha_0,\\
+\infty, & \text{for all } \alpha<\alpha_{0},
\end{cases}
\end{equation*}
uniformly with respect to $x\in\overline{\Omega}$.

\vspace{0.2cm}

\item \label{f2}
The function $f$ is sublinear at the origin, uniformly with respect to $x$, in the sense that
\begin{equation*}
\lim_{s\rightarrow 0}
\sup_{x\in \overline{\Omega}}
\left\vert
\frac{f(x,s)}{s}
\right\vert
=
0.
\end{equation*}

\vspace{0.2cm}

\item \label{f3}
There exist constants $s_0,K_0>0$ such that
\begin{equation*}
\vert F(x,s) \vert \leq K_0\vert f(x,s)\vert,
\quad
\text{for all } (x,s)\in \overline{\Omega}\times
\left((-\infty,-s_0]\cup [s_0,+\infty)\right),
\end{equation*}
where $F(x,s)=\int_0^s  f(x,\tau )\, \mathrm{d}\tau$.

\vspace{0.2cm}

\item \label{f4}
The following limit holds:
\begin{equation*}
\lim_{s\rightarrow +\infty}
\frac{s f(x,s)}{\exp(\alpha_0 s^{2})}
=
+\infty,
\end{equation*}
uniformly with respect to $x\in \overline{\Omega}$.

\vspace{0.2cm}

\item \label{f5}
For every $x\in\overline{\Omega}$ and every $s\in\mathbb R\setminus\{0\}$,
\begin{equation*}
0<2F(x,s)\leq s f(x,s).
\end{equation*}

\vspace{0.2cm}

\item \label{f6}
For every $x\in \overline{\Omega}$, the function
\begin{equation*}
s\mapsto \frac{f(x,s)}{\vert s \vert}
\end{equation*}
is strictly increasing on $(-\infty,0)$ and on $(0,+\infty)$.
\end{enumerate}
A typical example of a function satisfying assumptions
\ref{f1}-\ref{f6} is
$
f(x,s)=\vert s\vert^{p-2}s\exp(\alpha_0s^{2}),
$
where $p>2$.

\vspace{0.2cm}

We now state our first main result. The notions of ground state solution
for the original system \eqref{SC} and for the reduced problem
\eqref{nonlocalproblem} are specified in
Definitions~\ref{eq:defingroundstateforj} and~\ref{def:groundstate_reduced},
respectively.

\begin{theorem}\label{eq:mudouagoraC}
Suppose that assumptions \ref{f1}-\ref{f6} hold and that
$\delta,\gamma>0$. Then the nonlocal problem
\eqref{nonlocalproblem} admits a nontrivial ground state solution
$u\in H^{1}_{0}(\Omega)$. 
Moreover, if $v\in H^{1}_{0}(\Omega)$ is defined by
$
\delta v=B[u],
$
then $(u,v)$ is a nontrivial ground state solution of the system \eqref{SC}.
\end{theorem}

Theorem \ref{eq:mudouagoraC} concerns the local-minimum geometry of the energy associated with the reduced problem $(P)$ at the trivial solution $u=0$. Indeed, assumption \ref{f2} implies that $F(x,s)=o(s^2)$ as $s\to 0$, uniformly with respect to $x\in\Omega$. Consequently, the nonlinear contribution is of lower order than the positive quadratic part of the reduced energy near the origin, and $u=0$ is a strict local minimum. This local behavior provides the positive part of the mountain pass geometry, while the monotonicity assumption \ref{f6} allows the corresponding minimax level to be identified with the least energy on the Nehari manifold.

The second part of the paper addresses the complementary situation in which the trivial solution is no longer a local minimum of the reduced energy. The behavior of $F(x,s)$ near zero is instead related to a spectral gap of $\operatorname{T}=-\Delta+B$. This produces a finite-dimensional family of descending directions, while the energy remains positive on a sufficiently small sphere along the complementary directions. We refer to this configuration as the saddle-point geometry at the origin, and the associated critical point is obtained through a linking argument. Thus, the terms {\it local-minimum geometry} and {\it saddle-point geometry} refer throughout the paper to the local behavior, near $u=0$ in $H_0^1(\Omega)$, of the energy associated with the reduced problem \eqref{nonlocalproblem}. In particular, the two theorems concern distinct and incompatible local behaviors of the reaction term near the origin and should not be interpreted as a multiplicity result for the same nonlinearity.

For the saddle point result, we shall need additional assumptions involving the
eigenvalues $\widehat{\lambda}_j$ of $\operatorname{T}$ defined in
\eqref{eq:eigenvalues_intro}. We fix $k\in\mathbb{N}$ such that
$\lambda_k<\lambda_{k+1}$. Under the assumption
$\delta\leq\lambda_1+\gamma$, this implies that 
\begin{equation} \label{eq:lambda-gap} \tag{$\mathcal{G}_k$}
\widehat{\lambda}_k<\widehat{\lambda}_{k+1}.    
\end{equation}
The assumptions below are formulated with respect to this spectral gap and
provide the estimates required to establish the linking geometry. Moreover,
assumption \ref{f4} provides the lower growth at the critical exponential rate
needed to show that the corresponding minimax level lies below the
Trudinger-Moser threshold.

\begin{enumerate}[label=$(f_{\arabic*})$, ref=$(f_{\arabic*})$, start=7]

\item \label{f7}
There exist $\delta_0>0$ and
$\mu_*\in(\hat{\lambda}_k,\hat{\lambda}_{k+1})$ such that
\begin{equation*}
F(x,s)\leq \frac{\mu_*}{2}s^2,
\quad
x\in\overline{\Omega},\quad \vert s\vert\leq\delta_0.
\end{equation*}

\item \label{f8}
There exists
$\sigma\in(0,\mu_*-\hat{\lambda}_k)$ such that
\begin{equation*}
F(x,s)\geq
\frac{\hat{\lambda}_k+\sigma}{2}s^2,
\quad
x\in\overline{\Omega},\quad s\in\mathbb R.
\end{equation*}

\item \label{f9}
There exist $\delta_1>0$ and $C>0$ such that
\begin{equation*}
\vert f(x,s)\vert\leq C\vert s\vert,
\quad
x\in\overline{\Omega},\quad \vert s\vert\leq\delta_1.
\end{equation*}
\end{enumerate}
A typical example of a function satisfying the assumptions \ref{f1}, \ref{f3}-\ref{f5} and \ref{f7}-\ref{f9} is
\begin{equation*}
f(x,s)=a_0s+bs\left(\exp(\alpha_0s^2)-1-\alpha_0s^2\right),
\end{equation*}
where $b>0$ and $a_0$ is a positive constant satisfying
$
\widehat{\lambda}_k+\sigma<a_0<\mu_{*}<\widehat{\lambda}_{k+1}.
$

Our second main result can be stated as follows.

\begin{theorem}[Saddle point at the origin] \label{eq:teoremadoissaddle} Assume that $\delta \leq \lambda_{1}+\gamma$, so that the spectral gap condition \eqref{eq:lambda-gap} is satisfied, and that
\ref{f1}, \ref{f3}-\ref{f5},
\ref{f7}-\ref{f9} hold. Then the equation \eqref{nonlocalproblem} has a
nontrivial weak solution. Consequently, the system \eqref{SC} has a nontrivial
weak solution $(u,v)$ such that $\delta v=B[u]$.
\end{theorem}
The main contribution of this paper is to overcome the combined difficulties
produced by the strongly indefinite variational structure of FitzHugh-Nagumo
systems and the loss of compactness associated with critical exponential growth
in dimension two. By means of the classical reduction method, the original system
is transformed into a reduced nonlocal equation whose energy functional has a
positive definite quadratic part. The reduction, however, introduces a resolvent
term $B$, which induces an equivalent but nonstandard norm $\Vert \cdot \Vert_B$ on
$H^1_0(\Omega)$, and the critical regime must be analyzed with respect to this norm. Since
$\Vert u \Vert_B \ge \Vert \nabla u \Vert_{L^2(\Omega)}$, the supremum defining the Trudinger-Moser
inequality in the $B$-norm is taken over a strictly smaller set than in the
classical case, so that boundedness is inherited but sharpness is not. We prove
that the optimal constant remains exactly $4\pi$ (Lemma~\ref{eq:TMeditadaC}),
the sharpness resting on a renormalization of the concentrating Moser sequence in
the $B$-norm (Lemma~\ref{eq:lemadanorma0711C}). This is the analytical core on which both
existence results stand: the sharp value $4\pi$ is what locates the threshold
$2\pi/\alpha_0$ at which compactness is lost, and every minimax estimate in the
paper is calibrated against it. A different optimal constant would displace the
concentration scale relative to the critical exponent $\alpha_0$ fixed
in~\ref{f1}, and the entire compactness analysis would have to be reorganized
accordingly. To exploit this threshold along weakly convergent sequences, we
further establish a Lions-type compactness lemma adapted to the $B$-norm; see
Lemma~\ref{eq:lionsLemm}.

Both theorems use this sharp threshold, but the arguments diverge from there. In
the ground state case, the minimax level is preserved below $2\pi/\alpha_0$ through the
usual concentrating Moser sequence; see Lemma~\ref{eq:lemadonivelc}. In the saddle
point case, we use the spectral decomposition of $\operatorname{T}= -\Delta+B$ to construct a
linking geometry, and we control the critical level through a projected Moser
estimate in the $B$-norm compatible with the finite-dimensional spectral
splitting; see Lemma~\ref{lem:projected-moser-B}. This construction dispenses with the
monotonicity condition~\ref{f6} required in the ground state argument. In both
cases, the solution of the reduced problem is transferred back to the original
FitzHugh-Nagumo system through the resolvent correspondence $\delta v = B[u]$.

\vspace{0.2cm}

The paper is organized as follows. In Section \ref{sec2C}, we develop the
variational framework for system \eqref{SC}. We discuss the strongly indefinite
structure of the original system, derive the nonlocal reduction, and introduce
the functional setting induced by the resolvent operator. Particular attention is
given to the nonlocal norm on $H^1_0(\Omega)$ and to the asymptotic behavior of
Moser functions in this norm. In Section \ref{sec3C}, we prove the existence of a
ground state solution. We establish the mountain pass geometry, prove the
critical minimax estimate below the Trudinger-Moser threshold, recover the
compactness needed to obtain a nontrivial critical point, and then identify it as
a ground state solution. In Section \ref{sec4C}, we study the saddle point case.
We first establish the spectral estimates associated with $\operatorname{T} = -\Delta+B$, then prove
the linking geometry, obtain the projected Moser estimate, and finally verify the
Palais-Smale condition needed to establish the existence of a nontrivial solution in the saddle point case.

\section{The variational framework}\label{sec2C}
In this section, we introduce the fundamental concepts needed to solve System \eqref{SC} variationally.
Before passing to the nonlocal reduction, let us describe the variational
structure naturally associated with system \eqref{SC}. Formally, the energy
functional corresponding to \eqref{SC} is given by
\begin{equation}
\label{eq:functional-system}
\mathcal{J}(u,v)
=
\frac{1}{2}\int_{\Omega} \vert \nabla u \vert^2\,\mathrm{d}x
-
\frac{1}{2}\int_{\Omega} \vert \nabla v \vert^2\,\mathrm{d}x
-
\frac{\gamma}{2}\int_{\Omega} v^2\,\mathrm{d}x
+
\delta\int_{\Omega} uv\,\mathrm{d}x
-
\int_{\Omega} F(x,u)\,\mathrm{d}x,
\end{equation}
for $(u,v)\in H_0^1(\Omega)\times H_0^1(\Omega)$, where
$
F(x,s)=\int_0^s f(x, \tau)\, \mathrm{d} \tau.
$
It is standard to check, under our assumptions and by the Trudinger-Moser inequality, that the functional $\mathcal{J}$ is well defined and of class $C^{1}$.
The functional $\mathcal{J}$ is strongly indefinite, since the quadratic part has
opposite signs in the variables $u$ and $v$. However, its critical points
are precisely the weak solutions of system \eqref{SC}. Indeed, for every
$(\varphi,\psi)\in H_{0}^{1}(\Omega)\times H_{0}^{1}(\Omega)$, we have
\begin{align*}
\mathcal{J}'(u,v)(\varphi,\psi)
=&
\int_{\Omega} \nabla u \nabla \varphi\,\mathrm{d}x
-
\int_{\Omega} \nabla v \nabla \psi\,\mathrm{d}x
-
\gamma\int_{\Omega} v\psi\,\mathrm{d}x
+
\delta\int_{\Omega} v\varphi\,\mathrm{d}x
\nonumber\\
&
+
\delta\int_{\Omega} u\psi\,\mathrm{d}x
-
\int_{\Omega} f(x,u)\varphi\,\mathrm{d}x.
\end{align*}
Thus, $\mathcal{J}'(u,v)=0$ if and only if
\begin{align*}
\int_{\Omega} \nabla u \nabla \varphi\,\mathrm{d}x
+
\delta\int_{\Omega} v\varphi\,\mathrm{d}x
-
\int_{\Omega} f(x,u)\varphi\,\mathrm{d}x
=&0,
\\
\int_{\Omega} \nabla v \nabla \psi\,\mathrm{d}x
+
\gamma\int_{\Omega} v\psi\,\mathrm{d}x
-
\delta\int_{\Omega} u\psi\,\mathrm{d}x
=&0,
\end{align*}
for every $\varphi,\psi\in H_0^1(\Omega)$. Therefore, critical points of $\mathcal{J}$
are precisely weak solutions of \eqref{SC}.

\begin{definition} \label{eq:defingroundstateforj}
We say that a pair $(u,v)\in H_0^1(\Omega)\times H_0^1(\Omega)$ is a ground state solution of system \eqref{SC} if $(u,v)$ is a nontrivial
weak solution of \eqref{SC} and
\begin{equation*}
\mathcal{J}(u,v)
=
\inf
\left\{
\mathcal{J}(\varphi,\psi)
\, :  \,
(\varphi,\psi)\in H_0^1(\Omega)\times H_0^1(\Omega),\
(\varphi,\psi)\neq (0,0),\
\mathcal{J}'(\varphi,\psi)=0
\right\},
\end{equation*}
where $\mathcal{J}$ denotes the energy functional associated with system \eqref{SC}.
\end{definition}

We recall that, since $\gamma>0$, the operator $-\Delta+\gamma$, endowed with homogeneous Dirichlet boundary conditions, is invertible. More precisely, for each $u\in L^2(\Omega)$ there exists a unique solution $v\in H_0^1(\Omega)\cap H^2(\Omega)$ of
\begin{equation*}
\begin{cases}
-\Delta v+\gamma v=\delta u, & \mbox{in } \Omega,\\
v=0, & \mbox{on } \partial\Omega.
\end{cases}
\end{equation*}
Thus,
\begin{equation*}
v=\delta(-\Delta+\gamma)^{-1}u.
\end{equation*}
In particular, the resolvent operator is first defined as a bounded linear operator
\begin{equation*}
(-\Delta+\gamma)^{-1}:L^2(\Omega)\to H_0^1(\Omega)\cap H^2(\Omega).
\end{equation*}
Using the compact embedding $H_0^1(\Omega)\hookrightarrow L^2(\Omega)$, we can also consider $(-\Delta+\gamma)^{-1}$ as a compact linear operator from $L^2(\Omega)$ into itself. Moreover, this operator is self-adjoint and positive in $L^2(\Omega)$. Multiplying by $\delta$, we obtain
\begin{equation*}
\delta v=\delta^2(-\Delta+\gamma)^{-1}u \eqqcolon B[u].
\end{equation*}
This defines a bounded, linear, compact, and self-adjoint operator
\begin{equation*}
B : L^2(\Omega) \rightarrow L^2(\Omega).
\end{equation*}
Since $H_0^1(\Omega) \hookrightarrow L^2(\Omega)$, for every $u,w\in H_0^1(\Omega)$ the term
\begin{equation*}
\int_\Omega B[u]w\, \mathrm{d}x
\end{equation*}
is well defined. Moreover, for each $u \in L^{2}(\Omega) \setminus \{0\},$ we have
\begin{equation*}
\int_\Omega B[u]u\, \mathrm{d}x > 0,
\end{equation*}
that is, the operator $B$ is positive. In terms of the operator $B$, System \eqref{SC} is equivalent to the nonlocal problem \eqref{nonlocalproblem}, namely,
\begin{equation*}
\begin{cases}
-\Delta u+B[u]=f(x,u), & \mbox{in } \Omega,\\
u=0, & \mbox{on } \partial\Omega.
\end{cases}
\end{equation*}
Thus, $(u,v)\in H_{0}^{1}(\Omega)\times H_{0}^{1}(\Omega)$ is a weak solution of system \eqref{SC} if and only if $u$ is a weak solution of \eqref{nonlocalproblem} and $\delta v=B[u]$. The operator $B$ induces the norm
\begin{equation*}
\Vert u\Vert_B =
\left[
\int_{\Omega}  \vert \nabla u \vert^2 \; \mathrm{d}x + \int_{\Omega}B[u]u \, \mathrm{d}x
\right]^{1/2}.
\end{equation*}
This norm is equivalent to the usual Dirichlet norm on $H_0^1(\Omega)$. Indeed, the positivity of $B$ implies that
\begin{equation*}
\Vert u\Vert_B^2 \geq \Vert \nabla u\Vert_{L^2(\Omega)}^2.
\end{equation*}
On the other hand, the boundedness of $B$ on $L^2(\Omega)$, combined with H\"{o}lder's inequality and the Poincar\'e inequality, yields
\begin{equation*}
\Vert u\Vert_B^2
\leq \Vert \nabla u\Vert_{L^2(\Omega)}^2
+\Vert B\Vert_{\mathcal{L}(L^2(\Omega),L^2(\Omega))}\,\Vert u\Vert_{L^2(\Omega)}^2
\leq C\,\Vert \nabla u\Vert_{L^2(\Omega)}^2.
\end{equation*}
Hence $\Vert\cdot\Vert_B$ and the standard $H_0^1(\Omega)$-norm are equivalent. Moreover, this norm is induced by the inner product
\begin{equation*}
\langle u,w\rangle_B =
\int_{\Omega} \nabla u\nabla w \, \mathrm{d}x + \int_{\Omega} B[u]w \, \mathrm{d}x.
\end{equation*}

We first consider the case where $\Omega=B(0,1)\subset\mathbb{R}^2$ is the unit ball centered at the origin. For every $n\in\mathbb{N}$ with $n\geq 2$, we consider the Moser function
\begin{equation} \label{eq:moserfunctionC}
\widetilde{G}_n(x) = \frac{1}{\sqrt{2\pi}}
\begin{cases}
(\log n)^{1/2}, & \text{if } \vert x \vert \leq \tfrac{1}{n}, \\[6pt]
\dfrac{\log \tfrac{1}{\vert x\vert }}{(\log n)^{1/2}}, & \text{if } \tfrac{1}{n}\leq \vert x\vert \leq 1, \\[10pt]
0, & \text{if } \vert x\vert \geq 1.
\end{cases}   
\end{equation}
Recall that the Moser function $\tilde G_n$ satisfies (see \cite{moser1971sharp})
\begin{equation*}
\int_{\Omega} \vert \nabla \tilde G_n \vert^{2}\, \mathrm{d}x = 1.
\end{equation*}
In order to analyze the asymptotic behavior of the Moser functions with respect to the nonlocal norm induced by the operator $B$, we establish the following result:
\begin{lemma} \label{eq:lemadanorma0711C}
Assume that $\Omega = B(0,1)$ and let $\tilde G_n$ be the Moser functions. Then
\begin{equation*}
\Vert \tilde G_n \Vert_{B}^{2} = 1 + O\left(\frac{1}{\log(n)} \right)= 1 + o_n(1).
\end{equation*}
\end{lemma}
\begin{proof}
First, notice that
\begin{equation*}
\int_{\Omega} B[\tilde{G}_n] \tilde{G}_n \, \mathrm{d}x \leq \int_{\Omega} \vert B[\tilde{G}_n]\vert \, \vert \tilde{G}_{n} \vert \, \mathrm{d}x 
   \leq \Vert B[\tilde{G}_n] \Vert_{L^2(\Omega)} \, \Vert \tilde{G}_n \Vert_{L^2(\Omega)}.
\end{equation*}
\noindent Since $B\in \mathcal{L}(L^2(\Omega),L^2(\Omega))$, it follows that
\begin{equation}\label{estimB1}
\int_{\Omega} B[\tilde{G}_n] \tilde{G}_n \, \mathrm{d}x \leq \Vert B \Vert_{\mathcal L(L^2(\Omega),L^2(\Omega))} \, \Vert \tilde{G}_n \Vert_{L^2(\Omega)}^2.
\end{equation}
By polar coordinates, we have
\begin{equation}\label{estimB2}
\Vert \tilde{G}_{n} \Vert_{L^2(\Omega)}^2 = \int_{\Omega} \vert \tilde{G}_n \vert^2 \, \mathrm{d}x 
   = \int_{0}^{1} \int_{0}^{2\pi} (\tilde{G}_n(r))^2 \, r \, \mathrm{d}\theta \, \mathrm{d}r.
\end{equation}
We rewrite the integral as follows:
\begin{align}\label{estimB3}
 \int_{0}^{1} \int_{0}^{2\pi} \bigl(\tilde{G}_n(r)\bigr)^{2} \, r \, \mathrm{d}\theta \, \mathrm{d}r  = & \dfrac{1}{2 \pi}\int_{0}^{1/n} \int_{0}^{2\pi} r \; \log n  \, \mathrm{d}\theta \, \mathrm{d}r 
   + \dfrac{1}{2 \pi \log n}\int_{1/n}^{1} \int_{0}^{2\pi} \left(\log \tfrac{1}{r}\right)^{2} r \, \mathrm{d}\theta \, \mathrm{d}r \nonumber \\
     = & \dfrac{\log n}{2n^{2}} + \dfrac{1}{\log n}\int_{1/n}^{1}  \left(\log \tfrac{1}{r}\right)^{2} r \, \mathrm{d}r \nonumber \\  = & o_{n}(1) + \dfrac{1}{\log n}\int_{1/n}^{1} \left(\log \tfrac{1}{r}\right)^{2} r \, \mathrm{d}r. 
\end{align}
Making the change of variables $s = \log(1/r),$ we obtain $\exp(s) = 1/r,$ which implies $r = \exp(-s).$ Thus $\mathrm{d}r = - \exp(-s)\; \mathrm{d}s$ and $r \; \mathrm{d}r = - \exp(-2s) \; \mathrm{d}s.$ Therefore, it follows that
\begin{align*}
 \int_{0}^{1} \int_{0}^{2\pi} \bigl(\tilde{G}_n(r)\bigr)^{2} \, r \, \mathrm{d} \theta \, \mathrm{d}r  = & o_{n}(1) + \dfrac{1}{\log n}\int_{1/n}^{1} \left(\log \tfrac{1}{r}\right)^{2} r \, \mathrm{d}r \\  = & o_{n}(1) + \dfrac{1}{\log n} \int^{0}_{\log n} - s^{2} \exp(-2s) \; \mathrm{d}s \\  = & o_{n}(1) + \dfrac{1}{\log n} \int^{\log n}_{0} s^{2} \exp(-2s) \; \mathrm{d}s.
\end{align*}
Note that 
\begin{equation*}
    \int^{\log n}_{0} s^{2} \exp(-2s) \; \mathrm{d}s \rightarrow \int^{+ \infty}_{0} s^{2} \exp(-2s) \; \mathrm{d}s = \dfrac{1}{4}, \quad n \rightarrow + \infty,
\end{equation*}
and hence, 
\begin{equation}\label{estimB4}
   O\left(\frac1{\log n}\right) = \dfrac{1}{\log n} \int^{\log n}_{0} s^{2} \exp(-2s) \; \mathrm{d}s \rightarrow 0, \quad n \rightarrow + \infty.
\end{equation}
Therefore, combining \eqref{estimB1}-\eqref{estimB4}, we obtain
\begin{equation*}
\int_{\Omega} B[\widetilde G_n]\widetilde G_n\, \mathrm{d}x  =o_{n}(1).   
\end{equation*}
Since $\Vert \nabla\widetilde G_{n} \Vert_{L^{2}(\Omega)}^2=1$, we conclude that
\begin{equation*}
\Vert \widetilde G_{n}\Vert_{B}^{2} = 1 + o_{n}(1).
\end{equation*}
This completes the proof.
\end{proof}
If $\Omega \subset \mathbb{R}^{2}$ is an arbitrary open bounded domain, let $x_0 \in \Omega,$ $d >0$ be such that the open ball $B(x_0, d)$ is contained in $\Omega$ and $\tilde{G}_{n}$ as defined in \eqref{eq:moserfunctionC}. For each $n\in\mathbb{N}$, it is standard to verify that the function
\begin{equation} \label{eq:mosersequenced}
    G_n(x) \coloneqq \widetilde{G}_n \left( \frac{x-x_0}{d} \right), \text{ for all } x \in \Omega,
\end{equation}
belongs to $H_0^1(\Omega)$, $\Vert G_{n} \Vert^{2}_{H^{1}_{0}(\Omega)} = 1$, and the support of $G_n$ is contained in $B(x_0, d)$. By the change of variables $y=(x-x_0)/d$ and proceeding as in
Lemma~\ref{eq:lemadanorma0711C}, we have
\begin{equation*}
\Vert G_n \Vert_{L^2(\Omega)}^2 = d^2 \Vert \widetilde G_n \Vert_{L^2(B(0,1))}^2 = O\left(\frac{1}{\log(n)} \right).    
\end{equation*}
Therefore,
\begin{equation*}
 \int_\Omega B[G_n]G_n \, \mathrm{d}x = O\left(\frac{1}{\log(n)} \right),  
\end{equation*}
and hence
\begin{equation} \label{eq:normassintoticad}
\Vert G_n \Vert_{B}^{2}
=
\int_{\Omega} \vert \nabla G_n \vert^{2}\, \mathrm{d}x
+
\int_{\Omega} B[G_n]\,G_n\, \mathrm{d}x
=
1 + d^{2} o_{n}(1) = 1 + O\left(\frac{1}{\log(n)} \right) =1 + o_{n}(1). 
\end{equation} 

\vspace{0.3cm}

The next lemma is the sharp Trudinger-Moser inequality for the $B$-norm. Since $\Vert u \Vert_{B}\geq \Vert \nabla u \Vert_{L^{2}(\Omega)}$, the boundedness follows from the classical Trudinger-Moser (Theorem \ref{eq:TMclassicaC}). The sharpness of $4\pi$ does not follow directly, since the supremum is taken over a smaller set. Its proof requires the renormalization of the Moser functions provided by Lemma \ref{eq:lemadanorma0711C}.

\begin{lemma}[Sharp Trudinger-Moser inequality for the \(B\)-norm] \label{eq:TMeditadaC}
\label{lem:TM-B-norm}
For each fixed $u\in H^1_0(\Omega)$, one has
\[
    \int_\Omega \exp(\alpha u^2)\,\mathrm{d}x < +\infty
    \qquad\text{for every }\alpha>0.
\]
Moreover, there exists a constant $C>0$, independent of $u$, such that
\[
    \sup_{\substack{u\in H^1_0(\Omega)\\ \Vert u \Vert_B\le 1}}
    \int_\Omega \exp(\alpha u^2)\,\mathrm{d}x
    \le C|\Omega|,
    \qquad
    \text{for every } 0 < \alpha\le 4\pi .
\]
The constant $4\pi$ is sharp in the following sense: for every \(\alpha>4\pi\),
\[
    \sup_{\substack{u\in H^1_0(\Omega)\\ \Vert u \Vert_B\le 1}}
    \int_\Omega \exp(\alpha u^2)\,\mathrm{d}x
    =+\infty .
\]
\end{lemma}
\begin{proof} 
Let $u\in H_0^1(\Omega)$ and $\alpha>0$. Since
$\|\nabla\,\cdot\,\|_{L^2(\Omega)}$ and $\|\cdot\|_B$ are equivalent norms on
$H_0^1(\Omega)$, they induce the same topology, and $C_c^\infty(\Omega)$ is
dense in $H_0^1(\Omega)$ with respect to either of them. Choose
$\psi\in C_c^\infty(\Omega)$ such that
\begin{equation} \label{eq:choicephi}
\Vert \nabla(u-\psi) \Vert_{L^2(\Omega)}^2 < \frac{2\pi}{\alpha},    
\end{equation}
and set $w \coloneqq u-\psi$, so that $u=w+\psi$. If $\|\nabla w\|_{L^2(\Omega)}=0$,
then $w\equiv 0$ and $u=\psi\in L^\infty(\Omega)$, whence the integral is
trivially finite. Assume therefore $\|\nabla w\|_{L^2(\Omega)}>0$.

By Young's inequality
\[
  u^2=(w+\psi)^2 \le 2w^2+2\psi^2 .
\]
Since $\psi\in C_c^\infty(\Omega)$, we have
$M \coloneqq \|\psi\|_{L^\infty(\Omega)}<+\infty$, and therefore
\[
  \int_\Omega \exp(\alpha u^2)\, \mathrm{d}x
    \le \exp(2\alpha M^2) \int_\Omega \exp(2\alpha w^2)\, \mathrm{d}x .
\]
Define
\[
\overline{w}:=\frac{w}{\|\nabla w\|_{L^2(\Omega)}}.
\]
Then \(\|\nabla\overline{w}\|_{L^2(\Omega)}=1\). Moreover, setting
\[
\beta \coloneqq 2\alpha\|\nabla w\|_{L^2(\Omega)}^2,
\]
the choice of \(\psi\) in \eqref{eq:choicephi} implies that
\[
\beta<4\pi.
\]
Therefore, by the classical Trudinger-Moser inequality (Theorem \ref{eq:TMclassicaC}),
\[
\int_{\Omega}\exp(2\alpha w^2)\,\mathrm{d}x
=
\int_{\Omega}\exp\bigl(\beta\overline{w}^{\,2}\bigr)\,\mathrm{d}x
\leq C|\Omega|<+\infty.
\]
Consequently,
\[
\int_{\Omega}\exp(\alpha u^2)\, \mathrm{d}x \leq
C\exp\bigl(2\alpha\|\psi\|_{L^\infty(\Omega)}^2\bigr)|\Omega|
<+\infty.
\]
Since $\alpha>0$ was arbitrary, this proves the first assertion.

Now, let $u \in H^{1}_{0}(\Omega)$ be such that $\Vert u \Vert_{B} \leq 1.$ This implies that $\Vert \nabla u \Vert_{L^{2}(\Omega)} \leq  1.$ Then, by the Trudinger-Moser inequality (Theorem \ref{eq:TMclassicaC}), we have that
\begin{equation*}
    \int_{\Omega} \exp(\alpha u^{2})\,\mathrm{d}x \leq \sup \limits_{ \{u \in H^{1}_{0}(\Omega) \;  \vert  \; \Vert \nabla u \Vert_{L^{2}(\Omega)} \leq 1 \}} \int_{\Omega} \exp(\alpha u^{2}) \; \mathrm{d}x \leq C \vert \Omega \vert, \quad \text{for all } \; \alpha \leq 4 \pi.
\end{equation*}
Since $u \in H^{1}_{0}(\Omega)$ is arbitrary, it follows that
\begin{equation*}
    \sup \limits_{ \{u \in H^{1}_{0}(\Omega) \;  \vert  \; \Vert u \Vert_{B} \leq 1 \}} \int_{\Omega} \exp(\alpha u^{2}) \; \mathrm{d}x \leq C \vert \Omega \vert, \quad \text{for all } \; \alpha \leq 4 \pi.
\end{equation*}

We claim that the constant $\alpha = 4\pi$ is sharp, in the sense that, if $\alpha > 4\pi$, then
\begin{equation*}
\sup_{\Vert u \Vert_B \leq 1}
\int_{\Omega} \exp(\alpha u^2)\,\mathrm{d}x = +\infty.
\end{equation*}
In order to prove this claim, we consider the Moser sequence defined in \eqref{eq:mosersequenced}. It follows from \eqref{eq:normassintoticad} that
$
\Vert G_n \Vert_B = 1+o_n(1).
$
Define
$
H_n = \frac{G_n}{\Vert G_n \Vert_B}.
$
Then
$
\Vert H_n \Vert_B = 1.
$

Note that
\begin{equation*}
\int_{\Omega} \exp(\alpha H_n^2)\,\mathrm{d}x = \int_{\Omega} \exp\left(\frac{\alpha G_n^2}{1+o_{n}(1)}\right)\,\mathrm{d}x
\geq
\int_{B(x_{0},d/n)}
\exp\left(
\frac{\alpha G_n^{2}}{1 + o_{n}(1)}
\right)\,\mathrm{d}x.
\end{equation*}
On $B(x_{0},d/n) \subset \Omega \subset \mathbb{R}^{2}$, we have that
$
G_n^2(x)=\frac{\log n}{2\pi}.
$
Thus,
\begin{equation*}
\int_{\Omega} \exp(\alpha H_n^2)\,\mathrm{d}x
\geq
\int_{B(x_{0},d/n)}
\exp\left(
\frac{\alpha \log n}{2\pi(1+o_n(1))}
\right)\,\mathrm{d}x\geq
\int_{B(x_{0},d/n)}
n^{\frac{\alpha}{2\pi(1+o_n(1))}}\,\mathrm{d}x,
\end{equation*}
which implies that
\begin{equation*}
\int_{\Omega} \exp(\alpha H_n^2)\,\mathrm{d}x
\geq
n^{\frac{\alpha}{2\pi(1+o_n(1))}}
 \vert B(x_{0},d/n) \vert =\pi d^{2}
n^{\frac{\alpha}{2\pi (1+o_n(1))}-2}.
\end{equation*}
Finally, if $\alpha>4\pi$, then $\frac{\alpha}{2\pi}-2>0$ and therefore
\begin{equation*}
\int_{\Omega} \exp(\alpha H_n^2)\,\mathrm{d}x
\rightarrow +\infty, \text{ as } \; n \rightarrow + \infty.
\end{equation*}
This proves the sharpness of the constant \(4\pi\) and completes the proof.
\end{proof}

The functional $\mathcal{I}:H^1_0(\Omega)\rightarrow\mathbb{R}$ associated with problem \eqref{nonlocalproblem} is given by
\begin{equation*}
    \mathcal{I}(u) = \dfrac{1}{2} \Vert u \Vert_{B}^{2} - \int_{\Omega} F(x,u) \, \mathrm{d} x.
\end{equation*}
In view of the Trudinger-Moser inequality (see Lemma \ref{eq:TMeditadaC}) and assumptions \ref{f1}-\ref{f3} (or \ref{f9} in the saddle point at origin case) the functional is well defined. Using the standard arguments (see \cite{de1989lectures}) and Lemma \ref{eq:TMeditadaC} one can prove that $\mathcal{I} \in C^{1}(H^{1}_{0}(\Omega), \mathbb{R})$ with
\begin{equation*}
    \mathcal{I}'(u)w = \langle u , w \rangle_{B} - \int_{\Omega} f(x,u)w \, \mathrm{d} x,
\end{equation*}
and the critical points correspond to weak solutions of \eqref{nonlocalproblem}.

\begin{definition} \label{def:groundstate_reduced}
We say that $u\in H^1_0(\Omega)$ is a ground state solution of the reduced
problem \eqref{nonlocalproblem} if $u$ is a nontrivial weak solution of \eqref{nonlocalproblem} and
\begin{equation*}
    \mathcal{I}(u)=\inf\left\{
    \mathcal{I}(w): w\in H^1_0(\Omega),\ w\neq 0,\ \mathcal{I}'(w)=0
    \right\}.
\end{equation*}
\end{definition}

We now introduce a compactness refinement adapted to the norm $\Vert \cdot \Vert_{B}$.
It will be used to control exponential terms along weakly convergent sequences.
\begin{lemma}[Lions' compactness lemma for the $B$-norm] \label{eq:lionsLemm}
Let $(u_n)_{n \in \mathbb{N}}\subset H_0^1(\Omega)$ be a sequence such that
$\Vert u_n \Vert_B = 1$, $u_n \rightharpoonup u \neq 0$ weakly in $H_0^1(\Omega)$ and $\Vert u \Vert_B<1$. Then, for every
\begin{equation*}
    0 < p < \frac{4\pi}{1-\Vert u \Vert_B^2},
\end{equation*}
there holds
\begin{equation*}
    \sup_{n\in\mathbb{N}}
    \int_{\Omega}\exp(pu_n^2)\,\mathrm{d}x <+\infty .
\end{equation*}
\end{lemma}

\begin{proof}
Since $u_n\rightharpoonup u$ weakly in $H_0^1(\Omega)$ and $B$ is compact, we have $B[u_n]\rightarrow B[u]$, strongly in $L^2(\Omega)$. In view of H\"{o}lder's inequality and the compact embedding, we deduce 
\begin{align*}
    \int_{\Omega}(B[u_{n}]u_{n}-B[u]u)\,\mathrm{d}x = &
    \int_{\Omega}(B[u_n]u_n-B[u]u_n+B[u]u_n-B[u]u)\,\mathrm{d}x \\  = &
    \int_{\Omega}(B[u_n]-B[u])u_n\,\mathrm{d}x
    +
    \int_{\Omega}B[u](u_n-u)\,\mathrm{d}x\\
    \leq & \Vert B[u_n]-B[u]\Vert_{L^2(\Omega)}
    \Vert u_n\Vert_{L^2(\Omega)}+  \Vert B[u]\Vert_{L^2(\Omega)}
    \Vert u_n-u\Vert_{L^2(\Omega)},
\end{align*}
which implies that
\begin{equation*}
    \int_{\Omega}B[u_n]u_n\,\mathrm{d}x
    \rightarrow
    \int_{\Omega}B[u]u\,\mathrm{d}x, \quad \mbox{as } n\rightarrow + \infty.
\end{equation*}
Hence, since $\Vert u_n \Vert_{B}=1$, we have that
\begin{equation}\label{eqa}
    \Vert \nabla u_n \Vert_{L^2(\Omega)}^2
    \rightarrow
    1-\int_{\Omega}B[u]u\,\mathrm{d}x
    \coloneqq \ell, \quad \text{as } n \rightarrow + \infty.
\end{equation}
We claim that $\ell>0$. Indeed, if $\ell=0$, then $\Vert \nabla u_n \Vert_{L^2(\Omega)}\rightarrow 0$, as $n \rightarrow + \infty$. By the weak lower semicontinuity of the Dirichlet norm, we obtain
\begin{equation*}
    \Vert \nabla u \Vert_{L^2(\Omega)} \leq \liminf \limits_{n \rightarrow + \infty} \Vert \nabla u_{n}\Vert_{L^{2}(\Omega)} = 0,
\end{equation*}
and hence $u=0$, which contradicts the assumption $u\neq 0$. Therefore, $\ell>0$.

Define
$
    w_n \coloneqq
    u_n/\Vert \nabla u_n \Vert_{L^2(\Omega)}.
$
Thus,
$
    \Vert \nabla w_n \Vert_{L^2(\Omega)}=1
$
and $w_n \rightharpoonup w \coloneqq u/\sqrt \ell$, as $n \rightarrow + \infty$ weakly in $H^1_0(\Omega)$. Moreover,
$
    \Vert \nabla w \Vert_{L^2(\Omega)}^2
    =
    \Vert \nabla u \Vert_{L^2(\Omega)}^2/\ell.
$
Thus, from \eqref{eqa}, we have that
\begin{equation*}
    1-\Vert \nabla w \Vert_{L^2(\Omega)}^2
    =
    1-\frac{\Vert \nabla u \Vert_{L^2(\Omega)}^2}{\ell}
    =
    \frac{1-\Vert u \Vert_B^2}{\ell} > 0.
\end{equation*}
If
$
    0<p<4\pi/(1-\Vert u\Vert_B^2),
$
then
\begin{equation*}
    p \ell < \frac{4\pi \ell}{1-\Vert u\Vert_B^2}
    =
    \frac{4\pi}{1-\Vert \nabla w\Vert_{L^2(\Omega)}^2}.
\end{equation*}
Since $\Vert \nabla u_n \Vert_{L^2(\Omega)}^2\rightarrow \ell,$ as $n \rightarrow + \infty,$ choose $\tilde{p} > p \ell$ such that 
\begin{equation*}
    p\Vert \nabla u_n \Vert^{2} < \tilde{p} <
    \frac{4\pi}{1-\Vert \nabla w\Vert_{L^2(\Omega)}^2},
\end{equation*}
for $n\in\mathbb{N}$ sufficiently large, which implies that
\begin{equation*}
    pu_n^2
    =
    p\Vert \nabla u_n \Vert_{L^2(\Omega)}^2 w_n^2
    \leq
    \tilde{p} w_n^2,
\end{equation*}
for every $n\in\mathbb{N}$ large enough. By Lions' Compactness Lemma \cite[Theorem I.6 and Remark I.18]{Lions1985LimitCase} applied to $(w_n)_{n \in \mathbb{N}}$, we obtain, for some $n_0 \in \mathbb{N}$,

\begin{equation*}
\sup_{n\geq n_0}
\int_{\Omega}\exp(\tilde{p}w_n^2)\,\mathrm{d}x<+\infty.
\end{equation*}
Therefore,
\begin{equation*}
\sup_{n\geq n_0}
\int_{\Omega}\exp(pu_n^2)\,\mathrm{d}x<+\infty.
\end{equation*}
On the other hand, by the first assertion of Lemma \ref{lem:TM-B-norm}, for each $n=1,\ldots,n_0-1$,
\begin{equation*}
\int_{\Omega}\exp(pu_n^2)\,\mathrm{d}x<+\infty.
\end{equation*}
Since there are only finitely many such indices, we conclude that
\begin{equation*}
\sup_{n\in\mathbb{N}}
\int_{\Omega}\exp(pu_n^2)\,\mathrm{d}x<+\infty.
\end{equation*}
This completes the proof.
\end{proof}

\begin{proposition}\label{prop:ground_state_equivalence}
Let $\mathcal{I}$ be the reduced functional and let $\mathcal J$ be the functional associated
with system \eqref{SC}. If $u$ is a ground state solution of the reduced nonlocal
problem \eqref{nonlocalproblem} and $v\in H^{1}_{0}(\Omega)$ is defined by $\delta v=B[u],$ then
$(u,v)$ is a ground state solution of system \eqref{SC}.
\end{proposition}
\begin{proof}
We divide the proof into four steps.

\medskip

\noindent\textbf{Step 1: From the reduced problem to the original system.}

\medskip

Let $u$ be a ground state solution of the reduced problem \eqref{nonlocalproblem}. Then $u\neq 0$ and
$\mathcal{I}'(u)=0$. Hence,
\begin{equation} \label{eq:critical-reduced-problem}
    \int_{\Omega}\nabla u\nabla \eta\,\mathrm{d}x
    +
    \int_{\Omega}B[u]\eta\,\mathrm{d}x
    =
    \int_{\Omega}f(x,u)\eta\,\mathrm{d}x,
    \quad
    \text{for all } \eta \in H^1_0(\Omega).
\end{equation}
Consider the solution operator $L:H^1_0(\Omega)\rightarrow H^1_0(\Omega)$ associated with
the linear problem
\begin{equation} \label{eq:def-L-operator}
\begin{cases}
-\Delta v+\gamma v=\delta u, & \text{in }\Omega,\\
v=0, & \text{on }\partial\Omega.
\end{cases}
\end{equation}
Thus, for each $u\in H^1_0(\Omega)$, we write $L[u]=v$, where $v$ is the unique weak
solution of \eqref{eq:def-L-operator}. By the definition of $B$, we have
$
B[u]=\delta L[u].
$
By setting
$
v=L[u],
$
we have
$
    B[u]=\delta v.
$
Substituting into \eqref{eq:critical-reduced-problem},
we obtain
\begin{equation} \label{eq:first-equation-system-weak}
    \int_{\Omega}\nabla u\nabla \eta\,\mathrm{d}x
    +
    \delta\int_{\Omega}v\eta\,\mathrm{d}x
    =
    \int_{\Omega}f(x,u)\eta\,\mathrm{d}x,
    \quad
    \text{for all }  \eta\in H^1_0(\Omega).
\end{equation}
Moreover, since $v=L[u]$, we also have
\begin{equation} \label{eq:second-equation-system-weak}
    \int_{\Omega}\nabla v\nabla \xi\,\mathrm{d}x
    +
    \gamma\int_{\Omega}v\xi\,\mathrm{d}x
    =
    \delta\int_{\Omega}u\xi\,\mathrm{d}x,
    \quad
    \text{for all }  \xi\in H^1_0(\Omega).
\end{equation}
Therefore, \eqref{eq:first-equation-system-weak} and
\eqref{eq:second-equation-system-weak} show that $(u,v)$ is a weak solution of
system \eqref{SC}. Since $u\neq 0$, it follows that $(u,v)\neq(0,0)$. Thus,
$(u,v)$ is a nontrivial critical point of $\mathcal J$.

\medskip

\noindent\textbf{Step 2: Preservation of the energy under the resolvent correspondence.} 

\medskip

Taking $\xi=v$ in \eqref{eq:second-equation-system-weak}, we have
\begin{equation} \label{eq:test-second-equation-v}
    \int_{\Omega}\vert\nabla v\vert^2\,\mathrm{d}x
    +
    \gamma\int_{\Omega}v^2\,\mathrm{d}x
    =
    \delta\int_{\Omega}uv\,\mathrm{d}x.
\end{equation}
Using \eqref{eq:test-second-equation-v}, we obtain
\begin{align*}
    \mathcal J(u,v)
    &=
    \dfrac{1}{2}\int_{\Omega}\vert\nabla u\vert^2\,\mathrm{d}x
    -
    \dfrac{1}{2}\int_{\Omega}\vert\nabla v\vert^2\,\mathrm{d}x
    -
    \dfrac{\gamma}{2}\int_{\Omega}v^2\,\mathrm{d}x
    +
    \delta\int_{\Omega}uv\,\mathrm{d}x
    -
    \int_{\Omega}F(x,u)\,\mathrm{d}x
    \\
    &=
    \dfrac{1}{2}\int_{\Omega}\vert\nabla u\vert^2\,\mathrm{d}x
    +
    \dfrac{\delta}{2}\int_{\Omega}uv\,\mathrm{d}x
    -
    \int_{\Omega}F(x,u)\,\mathrm{d}x
    \\
    &=
    \dfrac{1}{2}\int_{\Omega}
    \left(
        \vert\nabla u\vert^2+B[u]u
    \right)\,\mathrm{d}x
    -
    \int_{\Omega}F(x,u)\,\mathrm{d}x
    \\
    &=
    \mathcal{I}(u).
\end{align*}

\medskip

\noindent\textbf{Step 3: System critical points satisfy the same resolvent relation.}

\medskip

Let $(\varphi,\psi)$ be any nontrivial critical point of $\mathcal J$. Then
$(\varphi,\psi)$ satisfies
\begin{equation} \label{eq:first-equation-critical-J}
    \int_{\Omega}\nabla\varphi\nabla\eta\,\mathrm{d}x
    +
    \delta\int_{\Omega}\psi\eta\,\mathrm{d}x
    =
    \int_{\Omega}f(x,\varphi)\eta\,\mathrm{d}x,
    \quad
    \text{for all } \eta\in H^1_0(\Omega),
\end{equation}
and
\begin{equation} \label{eq:second-equation-critical-J}
    \int_{\Omega}\nabla\psi\nabla\xi\,\mathrm{d}x
    +
    \gamma\int_{\Omega}\psi\xi\,\mathrm{d}x
    =
    \delta\int_{\Omega}\varphi\xi\,\mathrm{d}x,
    \quad
    \text{for all } \xi\in H^1_0(\Omega).
\end{equation}
By the uniqueness of the solution of the linear problem, equation
\eqref{eq:second-equation-critical-J} implies
$
    \psi=L[\varphi].
$
Consequently,
$
    B[\varphi]=\delta L[\varphi]=\delta\psi.
$
Substituting into
\eqref{eq:first-equation-critical-J}, we obtain
\begin{equation} \label{eq:critical-reduced-varphi}
    \int_{\Omega}\nabla\varphi\nabla\eta\,\mathrm{d}x
    +
    \int_{\Omega}B[\varphi]\eta\,\mathrm{d}x
    =
    \int_{\Omega}f(x,\varphi)\eta\,\mathrm{d}x,
    \quad
    \text{for all } \eta\in H^1_0(\Omega).
\end{equation}
Therefore,
$
\mathcal{I}'(\varphi)=0.
$

We now show that $\varphi\neq 0.$ Suppose, by contradiction, that
$\varphi=0$. Then \eqref{eq:second-equation-critical-J} becomes
\begin{equation*}
    \int_{\Omega}\nabla\psi\nabla\xi\,\mathrm{d}x
    +
    \gamma\int_{\Omega}\psi\xi\,\mathrm{d}x
    =
    0,
    \quad
    \text{for all } \xi\in H^1_0(\Omega).
\end{equation*}
Taking $\xi=\psi$, we have
\begin{equation*}
    \int_{\Omega}\vert\nabla\psi\vert^2\,\mathrm{d}x
    +
    \gamma\int_{\Omega}\psi^2\,\mathrm{d}x
    =
    0.
\end{equation*}
Since $\gamma>0$, it follows that $\psi=0$, which contradicts the fact that
$(\varphi,\psi)$ is nontrivial. Hence $\varphi\neq 0$. Therefore, $\varphi$ is
a nontrivial critical point of $\mathcal{I}$.

Finally, repeating the computation of Step 2 with $(\varphi,\psi)$ instead of
$(u,v)$, and using $\psi=L[\varphi]$, we obtain
\begin{equation} \label{eq:J-varphi-psi-I-varphi}
    \mathcal J(\varphi,\psi)
    =
    \mathcal J(\varphi,L[\varphi])
    =
    \mathcal{I}(\varphi).
\end{equation}

\medskip

\noindent\textbf{Step 4: Comparison of critical energy levels.} 

\medskip

Since $u$ is a ground state solution of the reduced problem and $\varphi$ is a
nontrivial critical point of $\mathcal{I}$, we have
$
    \mathcal{I}(u)\leq \mathcal{I}(\varphi).
$
Thus, Step 2 and
\eqref{eq:J-varphi-psi-I-varphi} imply that
\begin{equation*}
    \mathcal J(u,v)
    =
    \mathcal{I}(u)
    \leq
    \mathcal{I}(\varphi)
    =
    \mathcal J(\varphi,\psi).
\end{equation*}
Therefore, $(u,v)$ minimizes $\mathcal J$ among all nontrivial critical points
of $\mathcal J$. Hence $(u,v)$ is a ground state solution of system \eqref{SC}.
\end{proof}
The following technical lemma replaces the global Ambrosetti–Rabinowitz condition in our setting.

\begin{lemma} \label{eq:superquadratic-F-infty} Assume that \ref{f3} and \ref{f5} hold. 
Then, for every $\theta>2$, there exists $R_{\theta}>0$ such that
\begin{equation*}
\theta F(x,s)\leq sf(x,s),
\quad \text{for all }
x\in\overline{\Omega},\quad \vert s \vert \geq R_{\theta}.
\end{equation*}
Moreover, there exist constants $C_{1,\theta},C_{2,\theta}>0$ such that
\begin{equation*}
F(x,s)\geq C_{1,\theta} \vert s \vert^{\theta}-C_{2,\theta},
\quad \text{for all }
(x,s)\in\overline{\Omega}\times\mathbb{R}.
\end{equation*}
\end{lemma}
\begin{proof}
Indeed, by \ref{f5}, we have
\begin{equation*}
F(x,s)>0
\quad\text{and}\quad
s f(x,s)>0,
\quad
x\in\overline{\Omega},\quad s\neq0.
\end{equation*}
In particular, for $s\neq0$,
\begin{equation*}
\vert s \vert\, \vert f(x,s) \vert = s f(x,s).
\end{equation*}
Fix $\theta>2$. By \ref{f3}, there exist $s_0,K_0>0$ such that
\begin{equation*}
F(x,s)\leq K_0 \vert f(x,s) \vert,
\quad
x\in\overline{\Omega},\quad \vert s \vert \geq s_0.
\end{equation*}
Choose
$
R>\max\{s_0,\theta K_0\}.
$
Then, for every $x\in\overline{\Omega}$ and every $\vert s \vert \geq R$, we have
\begin{equation*}
\theta F(x,s)
\leq
\theta K_0 \vert f(x,s) \vert
\leq
\vert s \vert \, \vert f(x,s) \vert
=
s f(x,s).
\end{equation*}
Therefore,
\begin{equation}\label{eq:AR-infty-theta}
\theta F(x,s)\leq s f(x,s),
\quad
x\in\overline{\Omega},\quad \vert s \vert\geq R.
\end{equation}

Now fix $x\in\overline{\Omega}$. For $s\geq R$, define
\begin{equation*}
\Psi_x(s)=\frac{F(x,s)}{s^\theta}.
\end{equation*}
Since $F_s(x,s)=f(x,s)$, we get
\begin{equation*}
\Psi_x'(s)
=
\frac{s f(x,s)-\theta F(x,s)}{s^{\theta+1}}.
\end{equation*}
By \eqref{eq:AR-infty-theta}, $\Psi_x'(s)\geq0$ for every $s\geq R$. Hence $\Psi_x$ is
nondecreasing on $[R,+\infty)$, and so
\begin{equation*}
\frac{F(x,s)}{s^\theta}
\geq
\frac{F(x,R)}{R^\theta},
\quad
s\geq R.
\end{equation*}
Since $F(x,R)>0$ for every $x\in\overline{\Omega}$ and $F$ is continuous, compactness of
$\overline{\Omega}$ gives
\begin{equation*}
m_+ \coloneqq \min_{x\in\overline{\Omega}}F(x,R)>0.
\end{equation*}
Therefore,
\begin{equation}\label{eq:F-positive-side}
F(x,s)\geq \frac{m_+}{R^\theta}s^\theta,
\quad
x\in\overline{\Omega},\quad s\geq R.
\end{equation}

We argue similarly on the negative side. For $r\geq R$, define
\begin{equation*}
\Phi_x(r)=\frac{F(x,-r)}{r^\theta}.
\end{equation*}
Then
\begin{equation*}
\Phi_x'(r)
=
\frac{(-r)f(x,-r)-\theta F(x,-r)}{r^{\theta+1}}.
\end{equation*}
Using \eqref{eq:AR-infty-theta} with $s=-r$, we obtain $\Phi_x'(r)\geq0$ for every
$r\geq R$. Hence
\begin{equation*}
\frac{F(x,-r)}{r^\theta}
\geq
\frac{F(x,-R)}{R^\theta},
\quad
r\geq R.
\end{equation*}
Again, since $F(x,-R)>0$ and $F$ is continuous, there exists
\begin{equation*}
m_- \coloneqq \min_{x\in\overline{\Omega}}F(x,-R)>0.
\end{equation*}
Thus,
\begin{equation}\label{eq:F-negative-side}
F(x,s)\geq \frac{m_-}{R^\theta}\vert s\vert ^\theta,
\quad
x\in\overline{\Omega},\quad s\leq -R.
\end{equation}

Set
\begin{equation*}
C_{1,\theta} \coloneqq \min\left\{\frac{m_+}{R^\theta},\frac{m_-}{R^\theta}\right\}>0.
\end{equation*}
From \eqref{eq:F-positive-side} and \eqref{eq:F-negative-side}, we have
\begin{equation*}
F(x,s)\geq C_1 \vert s \vert^\theta,
\quad
x\in\overline{\Omega},\quad \vert s\vert \geq R.
\end{equation*}
Since $F(x,s)\geq0$ for $s\neq0$ and $F(x,0)=0$, choosing
$
C_{2,\theta}\coloneqq C_{1,\theta}R^\theta
$
we obtain, for $\vert s\vert <R$,
\begin{equation*}
C_{1, \theta} \vert s\vert ^\theta-C_{2,\theta}\leq0\leq F(x,s),
\end{equation*}
which finishes the proof.
\end{proof}

\section{Local-minimum geometry at the origin: ground state solutions}\label{sec3C}
\label{sec:positiveandnegativesecC}
In this section, we prove the existence of a nontrivial ground state solution of the reduced nonlocal Dirichlet problem \eqref{nonlocalproblem}, that is, 
\begin{equation*} 
\begin{cases}
-\Delta u + B[u]  =  f(x,u) & \text{in } \Omega, \\
u  =  0 & \text{on } \partial \Omega.
\end{cases}
\end{equation*}
The energy functional $\mathcal{I}: H^{1}_{0}(\Omega) \rightarrow \mathbb{R}$ is given by
\begin{equation*}
    \mathcal{I}(u) = \dfrac{1}{2} \Vert u \Vert_{B}^{2} - \int_{\Omega} F(x,u) \, \mathrm{d} x,
\end{equation*}
where, as before,
$
F(x,s)=\int_0^s f(x,\tau)\, \mathrm{d} \tau
$ is the primitive of $f.$ 

\subsection{Mountain pass geometry}

We first show that the problem admits a nontrivial solution. To this end, we verify that $\mathcal{I}$ satisfies the conditions of the Mountain Pass Theorem.

\begin{definition} A sequence $(u_{n})_{n \in \mathbb{N}} \subset H^{1}_{0}(\Omega)$ is called a \emph{Palais-Smale sequence} at level $c \in \mathbb{R}$ (denoted by $(PS)_c$) if
\begin{equation*}
    \mathcal{I}(u_n) \rightarrow c \quad \text{and} \quad \mathcal{I}'(u_n) \rightarrow 0, \quad \text{as } n \rightarrow + \infty.
\end{equation*}

\end{definition}
Our goal is to show that $\mathcal{I}$ possesses a critical value $c$. To this end, we employ the Mountain Pass Theorem of Ambrosetti \& Rabinowitz \cite[Theorem~2.1]{AMBROSETTI1973349}. This theorem provides a $(PS)_c$ sequence at the minimax level
\[
c \coloneqq \inf_{\Tilde{\gamma}\in\Gamma}\max_{t\in[0,1]}\mathcal{I}(\Tilde{\gamma}(t)),
\]
where
\[
\Gamma
\coloneqq
\{\Tilde{\gamma}\in C([0,1],H^1_0(\Omega)):\Tilde{\gamma}(0)=0,\ \mathcal{I}(\Tilde{\gamma}(1))<0\}.
\]
The compactness result below the critical Trudinger-Moser energy level will then imply the convergence of this sequence to a nontrivial critical point.

The next lemma establishes that the functional satisfies the Mountain Pass geometry.

\begin{lemma} \label{mountainpassC} Assume that \ref{f1}-\ref{f3} and \ref{f5} hold. Then, the functional $\mathcal{I}$ satisfies the following geometric conditions:
\begin{boldromanenum}
    \item $\mathcal{I}(0) = 0$ and there exist $r_{0} > 0$ and $\varepsilon > 0$ such that $\mathcal{I}(u) \geq \varepsilon,$ for $\Vert u \Vert_{B} = r_{0};$
    \item There exists $e\in H^1_0(\Omega)$ with $\Vert e\Vert_B  >r_0$ such that $\mathcal{I}(e)<0$.
\end{boldromanenum}
\end{lemma}
\begin{proof}
We first prove that $(i)$ is satisfied. Indeed, let \(\alpha>\alpha_0\), \(p>1\), and \(q>2\). By \ref{f1} and \ref{f2}, there exist constants \(C_1,C_2>0\), with \(C_1<1/2\), such that
\begin{equation*}
    \vert F(x,s) \vert \leq C_{1} \vert s \vert^{2} + C_{2} \vert s \vert^{q} \exp{(\alpha s^{2})}, \quad \text{for all }  (x,s) \in \overline{\Omega} \times \mathbb{R}.
\end{equation*} 
By H\"older's inequality and the Sobolev embeddings in dimension two, we obtain, for $\Vert u \Vert_{B} = r$,
\begin{equation*}
\int_\Omega \vert F(x,u) \vert \,\mathrm{d}x
\leq C_1 r^2 + C_2 r^q
\left(\int_\Omega \exp{(p\alpha u^2)}\,\mathrm{d}x\right)^{1/p}.    
\end{equation*}
Choosing $r_0 > 0$ small enough so that $p\alpha r_0^2<4\pi$, Lemma \ref{eq:TMeditadaC} gives 
\begin{equation*}
\mathcal{I}(u)\geq \frac{1}{2} \Vert u \Vert_B^2-C_1 \Vert u \Vert_{B}^{2}-C_2 \Vert u \Vert_B^q.    
\end{equation*}
For $0 < r_{0} < \min \left\{\left( \frac{1/2 - C_{1} }{C_{2}}\right)^{1/(q-2)}, \left(\frac{4 \pi}{p \alpha} \right)^{1/2}\right\}$ sufficiently small, this yields $\mathcal{I}(u)\geq \varepsilon>0$, whenever $\Vert u \Vert_{B} = r_{0}$.

We now prove item $(ii)$. Fix $\varphi\in C_c^\infty(\Omega)\setminus\{0\}$ and let $\theta>2$. By Lemma \ref{eq:superquadratic-F-infty}, there exist constants $C_{1,\theta},C_{2,\theta}>0$ such that
\begin{equation*}
F(x,s)\geq C_{1,\theta}\vert s\vert ^\theta-C_{2,\theta},
\quad \text{for all }
(x,s)\in\overline{\Omega}\times\mathbb{R}.
\end{equation*}
Therefore, for every $t>0$,
\begin{align*}
\mathcal{I}(t\varphi)
&=
\frac{t^2}{2}\Vert \varphi \Vert_B^2
-
\int_\Omega F(x,t\varphi)\,\mathrm{d}x\\
&\leq
\frac{t^2}{2} \Vert \varphi \Vert_B^2
-
C_{1,\theta}t^\theta
\int_\Omega \vert \varphi \vert^\theta\,\mathrm{d}x
+
C_{2,\theta}\vert \Omega \vert\\
&=
\frac{t^2}{2}\Vert \varphi \Vert_B^2
-
C_{1,\theta}t^\theta
\Vert \varphi \Vert_{L^\theta(\Omega)}^\theta
+
C_{2,\theta}\vert \Omega \vert.
\end{align*}
Since $\theta>2$ and $\varphi\neq 0$, the negative term of order $t^\theta$ dominates the quadratic term. Hence,
\begin{equation*}
\mathcal{I}(t\varphi)\rightarrow-\infty
\quad\text{as }t\rightarrow+\infty.
\end{equation*}
Consequently, there exists $t_0>0$ sufficiently large such that
\begin{equation*}
\Vert t_0\varphi \Vert_B>r_0
\quad\text{and}\quad
\mathcal{I}(t_0\varphi)<0.
\end{equation*}
Setting $e=t_0\varphi$, we conclude that $\Vert e \Vert_B>r_0$ and $\mathcal{I}(e)<0$. Thus, condition (ii) holds.
\end{proof}

\subsection{The critical mountain pass level estimate}

Due to the loss of compactness in the critical case, we need the following technical lemmas.
\begin{lemma}\label{eq:lemadonivelc}
Assume that conditions \ref{f3}-\ref{f5} hold. Then
\[
    c<\frac{2\pi}{\alpha_0}.
\]
\end{lemma}

\begin{proof}
Taking $\varphi=G_n$ in the proof of Lemma \ref{mountainpassC}(ii), we obtain
\[
    \mathcal{I}(tG_n)\rightarrow-\infty
    \quad\text{as }t\rightarrow+\infty.
\]
Consequently,
\[
    c\leq \max_{t>0}\mathcal{I}(tG_n)
\]
for every $n\in\mathbb{N}$. Hence, it is sufficient to prove that, for
some $n\in\mathbb{N}$,
\begin{equation}\label{eq:061101C}
    \max_{t>0}\mathcal{I}(tG_n)
    <\frac{2\pi}{\alpha_0}.
\end{equation}

Suppose, by contradiction, that \eqref{eq:061101C} does not hold.
Then, up to a subsequence, still denoted by $(G_n)_{n\in\mathbb{N}}$,
we have
\begin{equation}\label{eq:061102C}
    \max_{t>0}\mathcal{I}(tG_n)
    \geq\frac{2\pi}{\alpha_0}.
\end{equation}
Since the map $t\mapsto\mathcal{I}(tG_n)$ is continuous and tends to
$-\infty$ as $t\to+\infty$, there exists $t_n>0$ such that
\begin{equation}\label{eq:061103C}
    \mathcal{I}(t_nG_n)
    =\max_{t>0}\mathcal{I}(tG_n).
\end{equation}

Set
\[
    \varepsilon_n
    \coloneqq \int_\Omega B[G_n]G_n\,\mathrm{d}x.
\]
By the positivity of $B$ and the refined estimate for the scaled Moser
sequence (see Lemma \ref{eq:lemadanorma0711C} and \eqref{eq:normassintoticad}), there exists a constant $C_G \coloneqq \dfrac{d^{2} \Vert B \Vert_{\mathcal L(L^2(\Omega),L^2(\Omega))}}{4}>0$ such that
\[
    0<\varepsilon_n\leq\frac{C_G}{\log n}
\]
for all sufficiently large $n$. In particular,
\[
    \Vert G_n \Vert_B^2=1+\varepsilon_n.
\]
By condition \ref{f5}, we have $F(x,t_nG_n)\geq0$. Thus, by
\eqref{eq:061102C} and \eqref{eq:061103C},
\begin{align*}
    \frac{2\pi}{\alpha_0}
    &\leq \mathcal{I}(t_nG_n)\\
    &=\frac{t_n^2}{2} \Vert G_n \Vert_B^2
      -\int_\Omega F(x,t_nG_n)\,\mathrm{d}x\\
    &=\frac{t_n^2}{2}(1+\varepsilon_n)
      -\int_\Omega F(x,t_nG_n)\,\mathrm{d}x\\
    &\leq\frac{t_n^2}{2}(1+\varepsilon_n).
\end{align*}
Therefore,
\begin{equation}\label{eq:071102C}
    \frac{4\pi}{\alpha_0}
    \leq t_n^2(1+\varepsilon_n).
\end{equation}
Since $\varepsilon_n\to0$, we have $1+\varepsilon_n\leq2$ for all
sufficiently large $n$. Hence,
\[
    t_n^2\geq\frac{2\pi}{\alpha_0},
\]
and therefore $(t_n)_{n\in\mathbb{N}}$ is bounded away from zero. Since $t_n$ is a maximum point of the map
$t\mapsto\mathcal{I}(tG_n)$, we have
\[
    \left.\frac{\mathrm{d}}{\mathrm{d}t}
    \mathcal{I}(tG_n)\right|_{t=t_n}=0.
\]
Multiplying this identity by $t_n$, we obtain
\begin{equation}\label{eq:071103C}
\begin{aligned}
    t_n^2(1+\varepsilon_n)
    &=\int_\Omega f(x,t_nG_n)t_nG_n\,\mathrm{d}x\\
    &\geq
    \int_{B(x_0,d)}
    f(x,t_nG_n)t_nG_n\,\mathrm{d}x,
\end{aligned}
\end{equation}
where we used condition \ref{f5} and the fact that $t_nG_n\geq0$.

Let $B(x_0,d)\subset\Omega$ be the ball fixed in
\eqref{eq:mosersequenced}. Since condition \ref{f4} is stated as a
limit equal to $+\infty$, we may fix
\[
    \beta_0>
    \frac{4e^{2C_G}}{\alpha_0d^2}.
\]
Then, by \ref{f4}, there exists $s_0>0$ such that
\begin{equation}\label{eq:081101C}
    f(x,s)s
    \geq\beta_0\exp(\alpha_0s^2),
    \qquad
    x\in\Omega,\quad s\geq s_0.
\end{equation}

By the change of variables $x=x_0+dy$, we obtain from
\eqref{eq:071103C}
\begin{align*}
    t_n^2(1+\varepsilon_n)
    &\geq
    d^2\int_{B(0,1)}
    f\bigl(x_0+dy,t_n\widetilde G_n(y)\bigr)
    t_n\widetilde G_n(y)\,\mathrm{d}y\\
    &\geq
    d^2\int_{B(0,1/n)}
    f\left(
        x_0+dy,
        \frac{t_n}{\sqrt{2\pi}}(\log n)^{1/2}
    \right)
    \frac{t_n}{\sqrt{2\pi}}(\log n)^{1/2}
    \,\mathrm{d}y.
\end{align*}

Since $(t_n)_{n \in \mathbb{N}}$ is bounded away from zero,
\[
    \frac{t_n}{\sqrt{2\pi}}(\log n)^{1/2}
    \longrightarrow+\infty.
\]
Thus, for all sufficiently large $n$, estimate
\eqref{eq:081101C} applies on $B(0,1/n)$, and therefore
\begin{align}
    t_n^2(1+\varepsilon_n)
    &\geq
    \beta_0d^2
    \exp\left(
        \frac{\alpha_0t_n^2}{2\pi}\log n
    \right)
    |B(0,1/n)|
    \nonumber\\
    &=
    \beta_0\pi d^2
    \exp\left[
        2\left(
            \frac{\alpha_0t_n^2}{4\pi}-1
        \right)\log n
    \right].
    \label{eq:071101C}
\end{align}

We claim that $(t_n)_{n\in\mathbb{N}}$ is bounded. Indeed, suppose,
by contradiction, that $t_n\to+\infty$ along a subsequence. Taking
logarithms in \eqref{eq:071101C}, we obtain
\[
    \log\bigl(t_n^2(1+\varepsilon_n)\bigr)
    \geq
    \log(\beta_0\pi d^2)
    +
    2\left(
        \frac{\alpha_0t_n^2}{4\pi}-1
    \right)\log n.
\]
The right-hand side grows as $t_n^2\log n$, while the left-hand side
grows only logarithmically in $t_n$, which is impossible. Hence,
$(t_n)_{n \in \mathbb{N}}$ is bounded.

Up to a subsequence, there exists $t_0\geq0$ such that
$t_n\to t_0$. Passing to the limit in \eqref{eq:071102C}, we obtain
\[
    t_0^2\geq\frac{4\pi}{\alpha_0}.
\]
We claim that the reverse inequality also holds. Indeed, suppose that
\[
    t_0^2>\frac{4\pi}{\alpha_0}.
\]
Then there exists $c_0>0$ such that, for all sufficiently large $n$,
\[
    2\left(
        \frac{\alpha_0t_n^2}{4\pi}-1
    \right)\geq c_0.
\]
Consequently, the right-hand side of \eqref{eq:071101C} tends to
$+\infty$, while the left-hand side remains bounded, which is a
contradiction. Therefore,
\[
    t_0^2\leq\frac{4\pi}{\alpha_0},
\]
and hence
\begin{equation*}
    t_0^2=\frac{4\pi}{\alpha_0}.
\end{equation*}

We now return to \eqref{eq:071102C}. It follows that
\[
    \frac{\alpha_0t_n^2}{4\pi}
    \geq\frac{1}{1+\varepsilon_n}.
\]
Therefore,
\begin{align*}
    2\left(
        \frac{\alpha_0t_n^2}{4\pi}-1
    \right)\log n
    &\geq
    2\left(
        \frac{1}{1+\varepsilon_n}-1
    \right)\log n\\
    &=
    -\frac{2\varepsilon_n}{1+\varepsilon_n}\log n\\
    &\geq-2C_G.
\end{align*}
Combining this estimate with \eqref{eq:071101C}, we obtain
\[
    t_n^2(1+\varepsilon_n)
    \geq\beta_0\pi d^2e^{-2C_G}.
\]
Letting $n\to+\infty$ and recalling that
\[
    t_n^2\rightarrow\frac{4\pi}{\alpha_0}
    \qquad\text{and}\qquad
    \varepsilon_n\rightarrow0,
\]
we conclude that
\[
    \frac{4\pi}{\alpha_0}
    \geq\beta_0\pi d^2e^{-2C_G}.
\]
Hence,
\[
    \beta_0
    \leq\frac{4e^{2C_G}}{\alpha_0d^2},
\]
which contradicts the choice of $\beta_0$.

Therefore, \eqref{eq:061101C} holds for some $n\in\mathbb{N}$, and
consequently
\[
    c<\frac{2\pi}{\alpha_0}.
\]
This completes the proof.
\end{proof}

\begin{lemma} \label{eq:lemadecrescenteTHC}
Assume that condition \ref{f2} and \ref{f6} hold.
Then, for each $x\in\Omega$, the map
\begin{equation*}
s \mapsto \frac{1}{2} sf(x,s)-F(x,s)    
\end{equation*}
is increasing on $(0,+\infty)$ and decreasing on $(-\infty,0)$. In particular,
\begin{equation*}
\frac{1}{2} sf(x,s)-F(x,s)\geq 0,    
\end{equation*}
for every $s \neq 0$.
\end{lemma}

\begin{proof}
First, we consider the case $s > 0.$ If $0 < s < t,$ then
\begin{align*}
\dfrac{s}{2}f(x ,s) - F(x ,s)
&=  \dfrac{f(x ,s)}{2\vert s \vert} \vert s \vert s - F(x ,t) + \int_{s}^{t} \dfrac{f(x ,\tau)}{\vert \tau \vert} \vert \tau \vert \; \mathrm{d}\tau \\ 
& <  \dfrac{f(x ,t)}{2\vert t \vert} \vert s \vert s - F(x ,t) + \dfrac{f(x ,t)}{\vert t \vert} \left( \dfrac{ \vert t \vert t}{2} - \dfrac{ \vert s \vert s}{2} \right) \\ 
& =  \dfrac{t}{2} f(x ,t) - F(x ,t).
\end{align*}

\noindent On the other hand, if $s < t < 0$, then it follows from \ref{f6} that
\[
\dfrac{f(x ,s)}{\vert s \vert} < \dfrac{f(x ,t)}{\vert t \vert}.
\]
Thus, we have
\begin{align*}
\dfrac{s}{2}f(x ,s) - F(x ,s)
&=  \dfrac{f(x ,s)}{2\vert s \vert} \vert s \vert s - F(x ,t) + \int_{s}^{t} \dfrac{f(x ,\tau)}{\vert \tau \vert} \vert \tau \vert \; \mathrm{d}\tau \\ 
& >  \dfrac{f(x ,s)}{2\vert s \vert} \vert s \vert s - F(x ,t) + \dfrac{f(x ,s)}{\vert s \vert} \left( \dfrac{ \vert t \vert t}{2} - \dfrac{ \vert s \vert s}{2} \right) \\ 
& >  \dfrac{t}{2} f(x ,t) - F(x ,t).
\end{align*}
The last inequality follows from the estimate 
$$
\dfrac{f(x ,s)}{2\vert s \vert}  \vert t \vert t > \dfrac{f(x ,t)}{2\vert t \vert} \vert t \vert t.
$$

Finally, since $F(x,0)=0$, it follows from \ref{f2} that
\begin{equation*}
\lim_{s\rightarrow0}\left(\frac12sf(x,s)-F(x,s)\right)=0.    
\end{equation*}
The monotonicity proved above therefore gives
\begin{equation*}
\frac{1}{2}sf(x,s)-F(x,s)\geq0,    
\end{equation*}
for every $s\neq 0$.
\end{proof}
We introduce the Nehari set and the Nehari energy level
\begin{equation*}
    \mathcal{N}  = \{ u \in H^{1}_{0}(\Omega) \setminus \{0\} \; : \; \mathcal{I}'(u)u = 0 \} \quad \mbox{and} \quad c_{\mathcal{N}} \coloneqq \inf_{u\in\mathcal{N}} \mathcal{I}(u).
\end{equation*}
This level will be compared with the mountain pass level in order to identify
the ground state energy.

\begin{lemma} \label{eq:desigualdadedosniveis}
Assume that conditions \ref{f1}, \ref{f3}, and \ref{f5}-\ref{f6} hold. Then $c \leq c_{\mathcal N}.$
\end{lemma}

\begin{proof}
Let $u \in \mathcal{N}$ and define $h:[0,+\infty) \rightarrow \mathbb{R}$ by $h(t) = \mathcal{I}(tu)$. Note that
\begin{equation*}
    h'(t) = t \Vert u \Vert^{2}_{B} - \int_{\Omega} f(x,tu)u \; \mathrm{d}x.
\end{equation*}
Since $u \in \mathcal{N}$, we deduce
\begin{align*}
h'(t) 
&= t \int_{\Omega} f(x,u)u \; \mathrm{d}x
- \int_{\Omega} f(x,tu)u \; \mathrm{d}x \\ 
&= t \int_{\Omega} \left[f(x,u)u - \dfrac{f(x,tu)}{t}u \right] \; \mathrm{d}x \\
&= t \int_{\Omega} \left[
\dfrac{f(x,u)}{\vert u\vert}
-
\dfrac{f(x,tu)}{\vert tu\vert}
\right]\vert u\vert u \; \mathrm{d}x.
\end{align*}
Using assumption \ref{f6}, we conclude that $h'(t)\geq 0$ for $0<t<1$ and $h'(t)\leq 0$ for $t>1$. 
Therefore, for $u\in\mathcal{N}$, we have 
\begin{equation*}
    \mathcal{I}(u) = \max \limits_{t \geq 0} \mathcal{I}(tu)=\max \limits_{t \geq 0} h(t)=h(1).
\end{equation*}

Finally, we define $g: [0,1] \rightarrow H^{1}_{0}(\Omega)$ by $g(t) = tt_{0}u$, where $t_{0}$ is such that $\mathcal{I}(t_{0}u) < 0.$ Then, $g \in \Gamma$ and 
\begin{equation*}
    c \leq \max \limits_{t \in [0,1]} \mathcal{I}(g(t)) \leq \max \limits_{t \geq 0} \mathcal{I}(tu) = \mathcal{I}(u).
\end{equation*}
Since $u \in \mathcal{N}$ is arbitrary, it follows that $c \leq c_{\mathcal{N}}.$
\end{proof}

\begin{remark} \label{eq:remark12}
Let
\begin{equation*}
\mathcal{S} \coloneqq \left\{ u \in H_{0}^{1}(\Omega) : \mathcal{I}'(u)=0 \mbox{ and } u \neq 0 \right\}
\end{equation*}
and set $c_{\mathcal{S}} \coloneqq \inf \limits_{u \in \mathcal{S}} \mathcal{I}(u).$ By Definition \ref{def:groundstate_reduced}, $u_{0} \in S$ is a ground state solution of \eqref{nonlocalproblem} whenever
$
\mathcal{I}(u_{0}) = c_{\mathcal{S}}.
$
Since $\mathcal{S} \subset \mathcal{N}$, we have
$c_{\mathcal{N}} \leq c_{\mathcal{S}}.$ 
Moreover, by Lemma \ref{eq:desigualdadedosniveis}, $c \leq c_{\mathcal{N}}$. Hence
$
c \leq c_{\mathcal{N}} \leq c_{\mathcal{S}}.
$
Therefore, in order to obtain a ground state solution $u_{0}$ for problem \eqref{nonlocalproblem}, it is enough to show that there exists $u_{0} \in \mathcal{S}$ such that
$
\mathcal{I}(u_{0}) = c_{\mathcal{S}}.
$
\end{remark}

\subsection{Compactness below the Trudinger-Moser level}

From now on, our goal is to ensure that $\mathcal{I}$ satisfies the Palais–Smale condition at the level $c < 2 \pi/\alpha_{0}.$ 

\begin{lemma} \label{eq:limitacaoC} Suppose that \ref{f3} and \ref{f5} hold. 
Let $(u_{n})_{n \in \mathbb{N}}$ be a $(PS)_c$ sequence. Then, $(u_{n})_{n \in \mathbb{N}}$ is bounded.
\end{lemma}

\begin{proof}
Let $(u_n)_{n\in\mathbb{N}}\subset H_0^1(\Omega)$ be a $(PS)_c$ sequence. Fix $\theta>2$. By Lemma \ref{eq:superquadratic-F-infty}, there exists $R_\theta>0$ such that
\begin{equation*}
\theta F(x,s)\leq sf(x,s),
\quad
x\in\overline{\Omega},\quad \vert s\vert \geq R_\theta.
\end{equation*}
Define
\begin{equation*}
H_\theta(x,s)
\coloneqq
\frac{1}{\theta}sf(x,s)-F(x,s).
\end{equation*}
It follows that
\begin{equation*}
H_\theta(x,s)\geq 0,
\quad
x\in\overline{\Omega},\quad \vert s\vert \geq R_\theta.
\end{equation*}
Moreover, since $H_\theta$ is continuous on the compact set
$\overline{\Omega}\times[-R_\theta,R_\theta]$, there exists a constant
$M_\theta>0$ such that
\begin{equation*}
H_\theta(x,s)\geq -M_\theta,
\quad
(x,s)\in\overline{\Omega}\times\mathbb{R}.
\end{equation*}

Since $(u_n)_{n\in\mathbb{N}}$ is a $(PS)_c$ sequence, there exist \(C>0\) and a sequence \((\varepsilon_n)_{n\in\mathbb{N}}\) with
\(\varepsilon_n\to0\) such that
\begin{equation*}
\vert \mathcal{I}(u_n)\vert \leq C
\quad\text{and}\quad
 \vert \mathcal{I}'(u_n)u_n \vert 
\leq \varepsilon_n\Vert u_n\Vert_B.
\end{equation*}
Consequently,
\begin{align*}
C+\frac{\varepsilon_n}{\theta}\Vert u_n\Vert_B
&\geq
\mathcal{I}(u_n)-\frac{1}{\theta} \mathcal{I}'(u_n)u_n\\
&=
\left(\frac{1}{2}-\frac{1}{\theta}\right)\Vert u_n\Vert_B^2
+
\int_\Omega
\left[
\frac{1}{\theta}f(x,u_n)u_n-F(x,u_n)
\right]\mathrm{d}x\\
&=
\left(\frac{1}{2}-\frac{1}{\theta}\right)\Vert u_n\Vert_B^2
+
\int_\Omega H_\theta(x,u_n)\,\mathrm{d}x\\
&\geq
\left(\frac{1}{2}-\frac{1}{\theta}\right)\Vert u_n\Vert_B^2
-M_\theta\vert \Omega\vert .
\end{align*}
Thus,
\begin{equation*}
\left(\frac{1}{2}-\frac{1}{\theta}\right)\Vert u_n\Vert_B^2
\leq
C+M_\theta\vert \Omega\vert 
+\frac{\varepsilon_n}{\theta}\Vert u_n\Vert_B.
\end{equation*}
Since $\theta>2$, we have
\begin{equation*}
\left( \frac{1}{2}-\frac{1}{\theta} \right) >0.
\end{equation*}
Therefore, the preceding inequality implies that
$(u_n)_{n\in\mathbb{N}}$ is bounded in $H_0^1(\Omega)$.
\end{proof}
    
Let $(u_{n})_{n \in \mathbb{N}}$ be a $(PS)_c$ sequence. In view of Lemma \ref{eq:limitacaoC}, up to a subsequence, we have that
\begin{equation} \label{eq:convergenciaweakly2205}
\left\{
\begin{array}{l}
u_n \rightharpoonup u_0, \ \text{in} \ H_0^1(\Omega), \\[6pt]
u_n(x) \rightarrow u_0(x), \ \text{a.e. in } \Omega, \\[6pt]
u_n \rightarrow u_0, \ \text{in} \ L^{p}(\Omega), \ 1 \leq p < +\infty.
\end{array}
\right.
\end{equation}
The next lemmas show that the weak limit \(u_0\) is a nontrivial weak solution of \eqref{nonlocalproblem}. For this, we recall the following result, which can be found in \cite[Lemma~2.1]{de2014elliptic}:

\begin{lemma} \label{eq:lemaRUFDIJC}
Suppose $\Omega$ is a bounded domain in $\mathbb{R}^2$. 
Let $(u_{n})_{n \in \mathbb{N}}$ be a sequence in $L^1(\Omega)$ such that $u_n \rightarrow u$ in $L^1(\Omega)$ 
and let $f(x,s)$ be a continuous function. 
Then, $f(x,u_n) \rightarrow f(x,u)$ in $L^1(\Omega)$ provided that 
$f(x,u_n) \in L^1(\Omega)$ for all $n \in \mathbb{N}$ and 
\begin{equation*}
\int_{\Omega} \vert f(x,u_n)u_n \vert\,\mathrm{d}x \le C.
\end{equation*}
\end{lemma}

\begin{lemma} \label{eq:limitesu0C} Assume that \ref{f1}, \ref{f3} and \ref{f5} are satisfied, and let $(u_n)_{n \in \mathbb{N}}$ be a $(PS)_c$ sequence. Then the term
$
    \displaystyle \int_{\Omega} f(x,u_{n}) u_{n} \, \mathrm{d}x
$
is bounded. Moreover,
\begin{itemize}
    \item[(i)] $f(x,u_n)\rightarrow f(x,u_0)$ in $L^1(\Omega)$, as $n\rightarrow+\infty$; \vspace{0.1cm}
    \item[(ii)]  $\displaystyle \int_\Omega F(x,u_{n}) \; \mathrm{d}x \rightarrow \displaystyle \int_\Omega F(x,u_{0}) \; \mathrm{d}x,$ as $n \rightarrow + \infty.$
\end{itemize}
\end{lemma}

\begin{proof} 
Since $(u_{n})_{n \in \mathbb{N}}$ is a $(PS)_c$ sequence, it follows from Lemma \ref{eq:limitacaoC} that $(u_{n})_{n \in \mathbb{N}}$ is bounded in $H^{1}_{0}(\Omega)$. Moreover, by \eqref{eq:convergenciaweakly2205}, we have that $(u_{n})_{n \in \mathbb{N}}\subset L^{1}(\Omega)$ and
$u_{n} \rightarrow u_{0}$, strongly in $L^{1}(\Omega)$. Thus, the first hypothesis of Lemma \ref{eq:lemaRUFDIJC} is satisfied. On the other hand,
\begin{equation*}
   o_{n}(1) = \mathcal{I}'(u_{n})u_{n} = \Vert u_{n} \Vert^{2}_{B} - \int_{\Omega} f(x,u_{n})u_{n} \; \mathrm{d}x.
\end{equation*}

\noindent Since $(u_{n})_{n \in \mathbb{N}}$ is bounded in $H^{1}_{0}(\Omega)$ and the norm $\Vert \cdot \Vert_{B}$ is equivalent to the usual norm of $H^{1}_{0}(\Omega)$, it follows that
\begin{equation*}
    \left \vert \int_{\Omega} f(x,u_{n})u_{n} \; \mathrm{d}x \right \vert  =  \left \vert \Vert u_{n} \Vert^{2}_{B} - o_{n}(1) \right \vert \leq C.
\end{equation*}
Moreover, by \ref{f5}, we have $f(x,s)s>0$ for every $s \neq 0$. Hence,
\begin{equation*}
    \int_{\Omega} \left \vert f(x,u_{n})u_{n} \right \vert \; \mathrm{d}x
    =
    \int_{\Omega} f(x,u_{n})u_{n} \; \mathrm{d}x
    \leq C.
\end{equation*}

\noindent It remains to check that $f(x,u_{n}) \in L^{1}(\Omega)$, for every $n \in \mathbb{N}$. Indeed, by \ref{f1}, for every $\alpha>\alpha_{0}$ there exists $C_{\alpha}>0$ such that
\begin{equation*}
    \vert f(x,s)\vert  \leq C_{\alpha}\exp(\alpha s^{2}),
    \quad \text{for all } (x,s) \in \Omega \times \mathbb{R}.
\end{equation*}
Since $u_{n} \in H^{1}_{0}(\Omega)$, the first assertion of the Trudinger-Moser inequality (Lemma \ref{eq:TMeditadaC}), applied after a normalization of $u_n$, implies that
\begin{equation*}
    \int_{\Omega} \vert f(x,u_{n}) \vert \; \mathrm{d}x
    \leq
    C_{\alpha}\int_{\Omega}\exp(\alpha u_{n}^{2}) \; \mathrm{d}x
    <+\infty.
\end{equation*}
Thus, $f(x,u_{n}) \in L^{1}(\Omega)$, for every $n \in \mathbb{N}$. Hence, all hypotheses of Lemma \ref{eq:lemaRUFDIJC} are satisfied. Consequently,
$f(x,u_{n}) \rightarrow f(x,u_{0})$, strongly in $L^{1}(\Omega)$, as $n \rightarrow +\infty$. This proves item $(i)$.

We now prove $(ii).$ In view of assumptions \ref{f3} and \ref{f5}, one may deduce
\begin{equation*}
\vert F(x, u_{n})\vert \leq C_{0} + K_{0} \vert f(x, u_{n})\vert.
\end{equation*}
Define
\begin{equation*}
g_n(x)\coloneqq C_0+K_0\vert f(x,u_n(x))\vert \quad \mbox{and} \quad 
g(x)\coloneqq C_0+K_0\vert f(x,u_0(x))\vert.
\end{equation*}
In view of $(i)$ we deduce
\begin{align*}
\Vert g_n-g\Vert_{L^1(\Omega)}
&=
K_0\int_{\Omega}
\big\vert
\vert f(x,u_n)\vert-\vert f(x,u_0)\vert
\big\vert\,\mathrm{d}x \nonumber\\
&\leq
K_0\int_{\Omega}
\vert f(x,u_n)-f(x,u_0)\vert\,\mathrm{d}x \nonumber\\
&=
K_0\Vert f(x,u_n)-f(x,u_0)\Vert_{L^1(\Omega)}\rightarrow0, \quad \mbox{as } n\rightarrow + \infty,
\end{align*}
i.e., $g_n\rightarrow g$ in $L^1(\Omega)$. Thus, since $u_n(x) \rightarrow u_{0}(x)$ a.e. in $\Omega$, it follows from the Generalized Dominated Convergence Theorem that
\[
\lim_{n \rightarrow +\infty} \int_{\Omega} F(x, u_n) \mathrm{d}x 
= \int_{\Omega} F(x, u_0) \; \mathrm{d}x,
\]
which finishes the proof. 
\end{proof}

\begin{lemma} \label{eq:lemadiferentezero}
Suppose that \ref{f1}-\ref{f5} hold. Then, the weak limit $u_{0}$ is nontrivial.
\end{lemma}

\begin{proof}
Assume, by contradiction, that $u_{0} = 0.$ In view of Lemma \ref{eq:limitesu0C} $(ii)$, we have
\begin{equation} \label{eq:Fu0C}
    \int_{\Omega} F(x,u_{n}) \, \mathrm{d}x \rightarrow \int_{\Omega} F(x,u_{0}) \, \mathrm{d}x = 0.
\end{equation}
Since $\mathcal{I}(u_{n}) \rightarrow c$, it follows from Lemma \ref{eq:lemadonivelc} and \eqref{eq:Fu0C} that
\[
\dfrac{1}{2} \Vert u_{n} \Vert^{2}_{B} \rightarrow c <  \dfrac{2 \pi}{\alpha_{0}}, \quad \text{as } n \rightarrow + \infty.
\]
Thus, there exist $n_0 \in \mathbb{N}$ and $\beta > 0$ such that $\alpha_0 \Vert u_n \Vert^{2}_{B} < \beta < 4\pi$ for all $n \ge n_0$. Choose $q > 1$ close to $1$ and $\alpha > \alpha_0$ close to $\alpha_0$ so that $q\alpha \Vert u_n \Vert^{2}_{B} < \beta < 4\pi$ still holds, for $n$ sufficiently large. In view of assumptions \ref{f1} and \ref{f2}, there exist $C_1,C_2>0$ such that
\begin{equation} \label{eq:TH3C}
    \vert f(x,s) s\vert \leq C_{1} \vert s \vert^{2} + C_{2} \vert s \vert \exp(\alpha s^{2}),
\end{equation}
for all $s \in \mathbb{R}.$
Hence, by \eqref{eq:TH3C}, the Trudinger–Moser inequality (Lemma \ref{eq:TMeditadaC}), and Hölder's inequality, we have
\begin{align} \label{eq:argumentoanalogoC}
\left\vert \int_{\Omega} f(x, u_n) u_n \, \mathrm{d}x \right\vert
\nonumber    &\leq C_1 \int_{\Omega} \vert u_n \vert^2 \; \mathrm{d}x + C_2 \int_{\Omega} \vert u_n \vert \exp(\alpha u_n^2) \; \mathrm{d}x \nonumber\\
\nonumber    &\leq C_1 \Vert u_n \Vert_{L^{2}(\Omega)}^{2} + C_2 \Vert u_n \Vert_{L^{\frac{q}{q-1}}(\Omega)} \left( \int_{\Omega} \exp \left[ q\alpha \Vert u_n \Vert_{B}^2 (u_n / \Vert u_n \Vert_{B})^2 \right] \; \mathrm{d}x \right)^{\frac{1}{q}} \nonumber\\
    &\leq C_1 \Vert u_n \Vert_{L^{2}(\Omega)}^{2} + C_3 \vert \Omega \vert^{1/q} \Vert u_n \Vert_{L^{\frac{q}{q-1}}(\Omega)} \rightarrow 0, \quad \text{as } n \rightarrow + \infty.
\end{align}
Hence, it follows from \eqref{eq:argumentoanalogoC} that
\begin{equation*}
o_{n}(1) = \mathcal{I}'(u_{n})u_{n} = \Vert u_{n} \Vert_{B}^{2} + o_{n}(1).
\end{equation*}
Thus, we conclude that
\[
0 < c = \lim_{n \rightarrow + \infty} \mathcal{I}(u_{n}) = \dfrac{1}{2}\lim_{n \rightarrow + \infty} \Vert u_{n} \Vert^{2}_{B} = 0,
\]
which is not possible. Therefore, $u_{0} \neq 0.$
\end{proof}
\begin{lemma} \label{eq:lema050526} Suppose that \ref{f1}-\ref{f6} hold. Then, the weak limit $u_{0}$ satisfies
\begin{equation*}
\Vert u_{0} \Vert^{2}_{B} \geq \int_{\Omega} f(x,u_{0})u_{0} \; \mathrm{d}x.
\end{equation*}
\end{lemma}
\begin{proof}
Suppose, by contradiction, that
\begin{equation*}
    \Vert u_{0} \Vert^{2}_{B} <\int_{\Omega} f(x,u_{0}) u_{0} \; \mathrm{d}x.
\end{equation*}
Consider the function $g: [0,1] \rightarrow \mathbb{R}$ defined by $g(t) = \mathcal{I}'(tu_{0})u_{0}.$ Note that,
\begin{equation*}
    \mathcal{I}'(tu_{0})u_{0}  =  t \Vert u_{0} \Vert^{2}_{B} - \int_{\Omega} \dfrac{f(x,t u_{0})}{t u_{0}} tu_{0}^{2} \; \mathrm{d}x.
\end{equation*}
By condition \ref{f2}, we have
\begin{equation*}
    \lim \limits_{t \rightarrow 0^{+}} \dfrac{\mathcal{I}'(tu_{0})u_{0}}{t} = \Vert u_{0} \Vert^{2}_{B} > 0,
\end{equation*}
which implies that $\mathcal{I}'(tu_{0})u_{0} > 0,$ for $t$ sufficiently small. On the other hand, for $t = 1,$ we have
\begin{equation*}
    \mathcal{I}'(u_{0})u_{0} = \Vert u_{0} \Vert^{2}_{B} - \int_{\Omega} f(x,u_{0})u_{0} \; \mathrm{d}x < 0.
\end{equation*}
Thus, by the Intermediate Value Theorem, there exists $t_{0} \in (0,1)$ such that $\mathcal{I}'(t_{0} u_{0})u_{0} = 0.$ In view of Lemma \ref{eq:lemadiferentezero} we have $u_{0} \neq 0,$ then the set $\Omega_{0} = \{x \in \Omega \; : \; u_{0}(x) \neq 0 \}$ has positive measure. Furthermore, by assumption \ref{f6}, Lemma \ref{eq:lemadecrescenteTHC}, the semicontinuity of the norm and Fatou's lemma, we obtain
 \begin{align*}
     c & \leq  c_{\mathcal{N}} 
     \leq \mathcal{I}(t_{0} u_{0}) - \dfrac{1}{2} \mathcal{I}'(t_{0} u_{0})(t_{0} u_{0}) \\ & =  \int_{\Omega} \left( \frac{1}{2}  f(x , t_{0} u_{0}) t_{0} u_{0} - F(x, t_{0} u_{0}) \right)\; \mathrm{d}x \\ & <  \int_{\Omega} \left( \frac{1}{2}f(x, u_{0})u_{0}  - F(x,u_{0})\right) \; \mathrm{d}x \\ & \leq  \liminf_{n\rightarrow + \infty} \int_{\Omega} \left( \frac{1}{2}f(x, u_{n}) u_{n} - F(x,u_{n}) \right) \\ & \leq  \lim \limits_{n \rightarrow + \infty} \mathcal{I}(u_{n}) - \dfrac{1}{2} \mathcal{I}'(u_{n})u_{n} = c,
 \end{align*}
which is a contradiction.
\end{proof}

\begin{proposition}
\label{lem:compactness_critical_level}  Assume that \ref{f1}-\ref{f6} hold. Then, the functional $\mathcal{I}$ satisfies the $(PS)_c$ condition at the mountain pass level
\begin{equation*}
    c<\dfrac{2\pi}{\alpha_0}.
\end{equation*}
Consequently, equation \eqref{nonlocalproblem} admits a weak solution $u_0\in H^{1}_{0}(\Omega)$ such that $\mathcal{I}(u_0)=c$. 
\end{proposition}

\begin{proof}
In view of Lemma \ref{mountainpassC}, the functional $\mathcal{I}$ satisfies the mountain pass geometry. Thus, let $(u_n)_{n \in \mathbb{N}}$ be a $(PS)_c$ sequence for $\mathcal{I}$. In view of Lemmas \ref{eq:limitacaoC} and \ref{eq:lemadiferentezero}, there exists $u_0\in H^{1}_{0}(\Omega)\backslash\{0\}$ such that, up to a subsequence, $u_n \rightharpoonup u_0$ weakly in $H^{1}_{0}(\Omega)$ and $\Vert u_n\Vert_B\rightarrow \rho_0\geq0$. Since $u_0\neq0$, we have $\rho_0>0$. Thus, using Lemma \ref{eq:limitesu0C} $(ii)$ we obtain
\begin{equation}\label{eq:cstar_identity}
    c
    =
    \lim_{n\rightarrow + \infty} \mathcal{I}(u_n)
    =
    \frac{1}{2}\rho_0^2
    -
    \int_{\Omega}F(x,u_0)\; \mathrm{d}x.
\end{equation}
By the weak lower semicontinuity of the $B$-norm we have
\begin{equation}\label{eq2}
\mathcal{I}(u_0)
=
\frac{1}{2} \Vert u_0 \Vert_B^2-\int_\Omega F(x,u_0)\,\mathrm{d}x
\leq
\frac{1}{2}\rho_0^2-\int_\Omega F(x,u_0)\,\mathrm{d}x
=
c.    
\end{equation}
Hence either $\mathcal{I}(u_0)=c$ or $\mathcal{I}(u_0)<c$.
We shall exclude the second case. Assume, by contradiction, that $\mathcal{I}(u_0)<c$. Then it follows from \eqref{eq:cstar_identity} and \eqref{eq2} that
\begin{equation*}
    \frac{1}{2}\Vert u_0\Vert^{2}_{B}
    -
    \int_{\Omega}F(x,u_0)\; \mathrm{d}x
    <
    \frac{1}{2}\rho_0^2
    -
    \int_{\Omega}F(x,u_0)\; \mathrm{d}x,
\end{equation*}
which implies that $\Vert u_0\Vert^{2}_{B}<\rho_0^2.$ Moreover, using \eqref{eq:cstar_identity}, we have
\begin{equation} \label{eq:1205261}
    \rho_0^2-\Vert u_0\Vert^{2}_{B}
    =
    2c-2\mathcal{I}(u_0).
\end{equation}
We claim that $\mathcal{I}(u_0)\geq 0$. In fact, it follows from Lemma \ref{eq:lema050526} that $\mathcal{I}'(u_{0})u_{0} \geq 0.$ Consequently, it follows from Lemma \ref{eq:lemadecrescenteTHC} that
\begin{equation} \label{eq:1205262}
    \mathcal{I}(u_{0}) \geq \mathcal{I}(u_{0}) - \dfrac{1}{2} \mathcal{I}'(u_{0})u_{0} = \int_{\Omega}  \left[\frac{1}{2}f(x,u_{0})u_{0} - F(x, u_{0}) \right] \; \mathrm{d}x \geq 0.
\end{equation}
Then, it follows from Lemma \ref{eq:lemadonivelc}, \eqref{eq:1205261} and \eqref{eq:1205262} that
\begin{equation*}
    \rho_0^2
    \leq
    2c+\Vert u_0\Vert^{2}_{B}<\frac{4\pi}{\alpha_0}
    +
    \Vert u_0\Vert^{2}_{B}.
\end{equation*} 
Thus, we obtain
\begin{equation}\label{eq:rho_tm_bound}
    \alpha_0\rho_0^2
    <
    \frac{4\pi}{1-\dfrac{\Vert u_0\Vert^{2}_{B}}{\rho_0^2}}.
\end{equation}
Define $v_n=\frac{u_n}{\Vert u_n\Vert_{B}}$. Since $u_n\rightharpoonup u_0$ in $H_0^1(\Omega)$ and $\Vert u_n\Vert_{B} \rightarrow \rho_0$, we have $v_n \rightharpoonup v_{0} \coloneqq \frac{u_0}{\rho_0}$, weakly in $H_0^1(\Omega)$. Notice also that $\Vert v_n\Vert_{B}=1$ and $\Vert v_0\Vert_{B} = \frac{\Vert u_0\Vert_{B}}{\rho_0} < 1$.
Hence
\begin{equation*}
    1-\Vert v_0\Vert^{2}_{B}
    =
    1-\frac{\Vert u_0\Vert^{2}_{B}}{\rho_0^2}
    >0.
\end{equation*}
Thus, from \eqref{eq:rho_tm_bound}, there exists $\beta>0$ such that
\begin{equation}\label{eq:beta_choice}
    \alpha_0\rho_0^2
    <
    \beta
    <
    \frac{4\pi}{1-\Vert v_0\Vert^{2}_{B}}.
\end{equation}
Since $\Vert u_n\Vert^2_{B}\rightarrow \rho_0^2$, we choose $q>1$ close to $1$ and $\alpha>\alpha_0$ close to $\alpha_0$ satisfying
\begin{equation*}
    q\alpha\Vert u_n\Vert^{2}_{B}
    <
    \beta
    <
    \frac{4\pi}{1-\Vert v_0\Vert^{2}_{B}},
\end{equation*}
for $n \in \mathbb{N}$ sufficiently large. 
Therefore, applying Lemma \ref{eq:lionsLemm} to $(v_n)_{n \in \mathbb{N}}$, we obtain
\begin{equation}\label{eq:tm_uniform}
    \sup_{n\in\mathbb{N}}
    \int_{\Omega}
    \exp(q\alpha u_n^2)\; \mathrm{d}x
    =
    \sup_{n\in\mathbb{N}}
    \int_{\Omega}
    \exp\left(q\alpha\Vert u_n\Vert^{2}_{B} v_n^2\right)\; \mathrm{d}x
    <+\infty.
\end{equation}

We claim that
\begin{equation}\label{eq:nonlinear_limit_goal}
    \int_{\Omega} f(x,u_n)(u_n-u_0)\; \mathrm{d}x
    \rightarrow 0, \quad \mbox{as } n\rightarrow +\infty.
\end{equation}
Indeed, in view of \eqref{eq:TH3C}, H\"older's inequality, and \eqref{eq:tm_uniform}, arguing as in \eqref{eq:argumentoanalogoC}, we deduce
\begin{align*}
    \left \vert
    \int_{\Omega} f(x,u_n)(u_n-u_0)\; \mathrm{d}x
    \right \vert
    \leq 
    C_3\Vert u_n-u_0\Vert_{L^{2}(\Omega)}
    +
    C_3
    \left(
    \int_{\Omega}\exp(q\alpha u_n^2)\; \mathrm{d}x
    \right)^{1/q}
    \Vert u_n-u_0\Vert_{q'},
\end{align*}
where $q'=q/(q-1)$ is the conjugate exponent of $q$. Since $u_{n}\rightarrow u_0$, strongly in $L^p(\Omega)$ for every $1\leq p<+\infty$, we conclude that \eqref{eq:nonlinear_limit_goal} holds. On the other hand, since $\mathcal{I}'(u_n)\rightarrow 0$, we have
\begin{align*}
    o_n(1)
    =&
    \mathcal{I}'(u_n)(u_n-u_0)
    \\
    =&
    \int_{\Omega}\nabla u_n \nabla (u_n-u_0)\; \mathrm{d}x
    +
    \int_{\Omega}B[u_n](u_n-u_0)\; \mathrm{d}x
    \\
    &-
    \int_{\Omega}f(x,u_n)(u_n-u_0)\; \mathrm{d}x
    \\
    = &
    \int_{\Omega}
    \left(
    \vert \nabla u_n\vert^2+B[u_n]u_n
    \right)\; \mathrm{d}x
    -
    \int_{\Omega}
    \left(
    \nabla u_n\nabla u_0+B[u_n]u_0
    \right)\; \mathrm{d}x
    \\
    &-
    \int_{\Omega}f(x,u_n)(u_n-u_0)\; \mathrm{d}x\\
    = & \Vert u_n\Vert^2_{B}
    -
    \Vert u_0\Vert^2_{B}
    + o_{n}(1)
\end{align*} 
which implies that $\Vert u_n\Vert^2_{B} \rightarrow \Vert u_0\Vert^{2}_{B}$, as $n\rightarrow + \infty$. Since $\Vert u_n\Vert^{2}_{B}\rightarrow \rho_0^2,$ we have that $\rho_0^2=\Vert u_0\Vert^{2}_{B}$, which contradicts $\Vert u_0\Vert^{2}_{B}<\rho_0^2.$ Therefore, the case $\mathcal{I}(u_0)<c$ cannot occur and $\mathcal{I}(u_0)=c$.
Thus, it follows from \eqref{eq2} that $\rho_{0}^{2}= \Vert u_{0} \Vert_B^2.$
Finally, since $u_n\rightharpoonup u_0$ in $H^1_0(\Omega)$ and $\Vert u_n \Vert_B \rightarrow \Vert u_0 \Vert_B$, we conclude that $u_n\rightarrow u_0,$ as $n \rightarrow + \infty$ strongly in $H^1_0(\Omega)$. 
Since $\mathcal{I}'(u_n)\rightarrow 0$ and $\mathcal{I}'$ is continuous, it follows that $\mathcal{I}'(u_0)=0.$
The proof is complete.
\end{proof}

\subsection{Proof of Theorem \ref{eq:mudouagoraC}}


\vspace{0.3cm}

In view of Proposition \ref{lem:compactness_critical_level}, there exists a nontrivial weak solution $u_0\in H^{1}_{0}(\Omega)$ at the level $c$. It remains only to prove that $u_0$ is a ground state solution. Since $\mathcal{I}'(u_0)=0$ and $\mathcal{I}(u_0)=c$, it follows from Remark~\ref{eq:remark12} that
\[
c\leq c_{\mathcal N}\leq c_{\mathcal{S}}\leq \mathcal{I}(u_0)=c.
\]
Therefore, $\mathcal{I}(u_0)=c_{\mathcal{S}}$, i.e., $u_0$ is a ground state solution of the reduced nonlocal problem \eqref{nonlocalproblem}. Then, in light of Proposition \ref{prop:ground_state_equivalence}, setting \(v_0=L[u_0]\), equivalently \(\delta v_0=B[u_0]\), we conclude that \((u_0,v_0)\) is a ground state solution of system \eqref{SC}.

\section{Saddle-point geometry at the origin: linking solutions} \label{sec4C}

In this section, we prove Theorem~\ref{eq:teoremadoissaddle}. Unlike the ground state result obtained in the previous section, the present argument does not rely on a Nehari-type characterization of the minimax level. Instead, we use the spectral decomposition associated with the nonlocal operator $\operatorname{T} = -\Delta+B$ to construct a linking geometry. The main point is to combine this linking structure with a critical minimax estimate below the Trudinger-Moser threshold. Since the
linking set contains a finite-dimensional component, the usual Moser functions have to be projected onto the complementary subspace. This leads to a projected Moser estimate in the $B$-norm, which is the key ingredient for proving that the linking level lies below $2\pi/\alpha_0$.
\subsection{Spectral estimates}
We begin by recalling the spectral structure of the nonlocal operator
\begin{equation*}
    \operatorname{T}=-\Delta+B.
\end{equation*}
This operator was studied by de Figueiredo and Mitidieri
\cite{Figueiredo1986Maximum} in their analysis of FitzHugh-Nagumo type
elliptic systems. In particular, if $(\lambda_j)_{j\in\mathbb N}$ denotes the
sequence of Dirichlet eigenvalues of $-\Delta$ in $\Omega$, counted with
multiplicity, and if $(\varphi_j)_{j\in\mathbb N}$ is a corresponding
$L^2(\Omega)$-orthonormal basis of eigenfunctions, then the same functions
$\varphi_j$ diagonalize the operator $\operatorname{T}=-\Delta+B$. More precisely, since
\begin{equation*}
    -\Delta \varphi_j=\lambda_j\varphi_j
    \quad\text{and}\quad
    B[\varphi_j]=\frac{\delta^2}{\lambda_j+\gamma}\varphi_j,
\end{equation*}
we have
\begin{equation*}
    \operatorname{T}(\varphi_j)
    =
    \left(\lambda_j+\frac{\delta^2}{\lambda_j+\gamma}\right)\varphi_j.
\end{equation*}
Therefore, the eigenvalues of $\operatorname{T}$ are given by
\begin{equation}\label{eq:eigenvalues_T_section4}
    \widehat{\lambda}_j
    =
    \lambda_j+\frac{\delta^2}{\lambda_j+\gamma},
    \quad j\in\mathbb N.
\end{equation}
Thus, although the operator $\operatorname{T}$ is nonlocal, its eigenspaces are generated by
the Dirichlet eigenfunctions of the Laplacian. Consequently, the finite-dimensional
space used in the linking construction may be defined only in terms of the
Laplacian eigenfunctions. For each $k\in\mathbb N$, we set
\begin{equation*}
    \mathcal{V}_k=\operatorname{span}\{\varphi_1,\ldots,\varphi_k\}.
\end{equation*}
We denote by $\mathcal{W}_k$ the orthogonal complement of $\mathcal{V}_k$ in $H^1_0(\Omega)$ with
respect to the inner product induced by the $B$-norm, namely
\begin{equation*}
    \mathcal{W}_k= \mathcal{V}_k^{\perp_B}.
\end{equation*}
Hence
\begin{equation*}
    H^1_0(\Omega)=\mathcal{V}_k\oplus \mathcal{W}_k.
\end{equation*}
In what follows, we assume that
\begin{equation*}
    \delta\leq \lambda_1+\gamma.
\end{equation*}
Under this condition, the function
\begin{equation*}
    t\mapsto t+\frac{\delta^2}{t+\gamma}
\end{equation*}
is nondecreasing for $t\geq \lambda_1$. Therefore, since
$(\lambda_j)_{j\in\mathbb N}$ is nondecreasing, the sequence
$(\widehat{\lambda}_j)_{j\in\mathbb N}$ is also nondecreasing.
\begin{lemma} 
Let $(\varphi_j)_{j\in\mathbb N}$ be the eigenfunctions of $-\Delta$ in $H_0^1(\Omega)$, chosen to be orthonormal in $L^2(\Omega)$, with corresponding eigenvalues $(\lambda_j)_{j\in\mathbb N}$. Then $(\varphi_j)_{j\in\mathbb N}$ is an orthogonal family with respect to the inner product $\langle\cdot,\cdot\rangle_B$. More precisely,
\begin{equation*}
\langle \varphi_i,\varphi_j\rangle_B
=
\hat{\lambda}_i\delta_{ij},
\quad i,j\in\mathbb{N}.
\end{equation*}
In particular,
\begin{equation*}
\Vert \varphi_i\Vert_B^2=\hat{\lambda}_i,
\quad i\in\mathbb N.
\end{equation*}
\end{lemma}

\begin{proof}
Since $-\Delta\varphi_i=\lambda_i\varphi_i$, we have
\begin{equation*}
\int_{\Omega}\nabla\varphi_i\nabla\varphi_j\,\mathrm{d}x
=
\lambda_i\int_{\Omega}\varphi_i\varphi_j\,\mathrm{d}x.
\end{equation*}
Moreover,
\begin{equation*}
B[\varphi_i]
=
\frac{\delta^2}{\lambda_i+\gamma}\varphi_i.
\end{equation*}
Therefore,
\begin{align*}
\langle \varphi_i,\varphi_j\rangle_B
&=
\int_{\Omega}\nabla\varphi_i\nabla\varphi_j\,\mathrm{d}x
+
\int_{\Omega}B[\varphi_i]\varphi_j\,\mathrm{d}x\\
&=
\left(
\lambda_i+\frac{\delta^2}{\lambda_i+\gamma}
\right)
\int_{\Omega}\varphi_i\varphi_j\,\mathrm{d}x\\
&=
\hat{\lambda}_i
\int_{\Omega}\varphi_i\varphi_j\,\mathrm{d}x.
\end{align*}
Since the eigenfunctions are orthonormal in $L^2(\Omega)$, it follows that
\begin{equation*}
\int_{\Omega}\varphi_i\varphi_j\,\mathrm{d}x
=
\delta_{ij}.
\end{equation*}
Consequently,
\begin{equation*}
\langle \varphi_i,\varphi_j\rangle_B
=
\hat{\lambda}_i\delta_{ij}.
\end{equation*}
Taking $i=j$, we also obtain
\begin{equation*}
\Vert \varphi_i\Vert_B^2
=
\hat{\lambda}_i.
\end{equation*}
The proof is complete.
\end{proof}
\begin{lemma} \label{eq:spectralinequalities}
Suppose that $\delta \leq \lambda_1 + \gamma$. Then the following inequalities hold:
\begin{boldromanenum}
    \item $\Vert u \Vert_{B}^{2} \leq \hat{\lambda}_{k} \Vert u \Vert^{2}_{L^{2}(\Omega)}$, for all $u \in \mathcal{V}_{k};$
    \item $\Vert u \Vert_{B}^{2} \geq \hat{\lambda}_{k + 1} \Vert u \Vert_{L^{2}(\Omega)}^{2},$ for all $u \in \mathcal{W}_{k}.$
\end{boldromanenum}
\end{lemma}
\begin{proof}
Let $u\in \mathcal{V}_k$. Then there exist $a_1,\ldots,a_k\in\mathbb R$ such that
\begin{equation*}
u=\sum_{j=1}^{k}a_j\varphi_j.
\end{equation*}
Moreover, since $(\varphi_j)_{j\in\mathbb N}$ is an $L^2(\Omega)$-orthonormal basis, we have
\begin{equation*}
\Vert u \Vert_{L^{2}(\Omega)}^{2}
=
\sum_{j=1}^{k}a_j^{2}.
\end{equation*}
As $B=\delta^2(-\Delta+\gamma)^{-1}$, it follows that
\begin{equation*}
B[\varphi_{j}]
=
\frac{\delta^2}{\lambda_j+\gamma}\varphi_j,
\quad j \in \mathbb{N}.
\end{equation*}
Since the eigenfunctions are orthonormal in $L^2(\Omega)$, we obtain
\begin{align*}
\Vert u\Vert_B^2
&=
\int_\Omega \vert \nabla u \vert ^2\, \mathrm{d}x+\int_\Omega B[u]u\,\mathrm{d}x\\
&=
\sum_{j=1}^{k}
\left(\lambda_j+\frac{\delta^2}{\lambda_j+\gamma}\right)a_j^2\\
&=
\sum_{j=1}^{k}\widehat{\lambda}_j a_j^2\\
&\leq
\widehat{\lambda}_k\sum_{j=1}^{k}a_j^2\\
&=
\widehat{\lambda}_k\Vert u\Vert_{L^2(\Omega)}^2.
\end{align*}
On the other hand, if $u\in \mathcal{W}_k$, then
\[
u=\sum_{j=k+1}^{+\infty}a_j\varphi_j.
\]
Therefore,
\begin{align*}
\Vert u\Vert_B^2
&=
\sum_{j=k+1}^{+\infty}\widehat{\lambda}_j a_j^2\\
&\geq
\widehat{\lambda}_{k+1}\sum_{j=k+1}^{+\infty}a_j^2\\
&=
\widehat{\lambda}_{k+1}\Vert u\Vert_{L^2(\Omega)}^2.
\end{align*}
The proof is complete.
\end{proof}

\subsection{Linking geometry}
In this subsection, we establish the linking structure associated with the
splitting
\begin{equation*}
H^1_0(\Omega)=\mathcal{V}_k\oplus \mathcal{W}_k.
\end{equation*}
The role of the finite-dimensional space $V_k$ is to describe the directions in
which the quadratic part of the functional is dominated by the potential term,
whereas the complementary space $W_k$ provides the positive direction near the
origin. This is precisely the variational mechanism behind the saddle point
geometry.

More specifically, we shall prove that the functional $\mathcal{I}$ is positive on a small
sphere contained in $W_k$ and nonpositive on the boundary of a suitable
finite-dimensional cylinder built over $V_k$. These two estimates allow us to
apply a linking theorem and obtain a Palais-Smale sequence at a positive
minimax level. The main difficulty, which will be addressed in the next
subsection, is to show that this level lies below the critical
Trudinger-Moser threshold $2\pi/\alpha_0$. 
\begin{theorem}(Linking Theorem ~\cite[Theorem~1.2, p.~162]{RABINOWITZ1978161}.) \label{linkingtheorem}
Let $\mathcal{I} \in C^1(H_0^1(\Omega),\mathbb R)$ and suppose that $H_0^1(\Omega)=\mathcal V_k\oplus\mathcal W_k$, where $\dim \mathcal V_k<+\infty$. Assume that the following conditions hold:
\begin{boldromanenum}
    \item there exist $a>0$ and $\rho>0$ such that $ \mathcal{I}(u)\geq a,
    \; u\in\mathcal W_k,\; \Vert u\Vert_B=\rho;$

    \item there exist $z\in\mathcal W_k$, with $\Vert z\Vert_B=1$, and $R>\rho$ such that, if
    \begin{equation*}
    Q
    =
    \left\{
    v+tz:\ v\in\mathcal V_k,\ \Vert v\Vert_B\leq R,\ 0\leq t\leq R
    \right\},
    \end{equation*}
    then $ \mathcal{I}(u)\leq0,
    \;
    u\in\partial Q.$
\end{boldromanenum}
Let $\Gamma
=
\left\{
\eta\in C(Q,H_0^1(\Omega)):\ \eta(u)=u \ \text{for all } u\in\partial Q
\right\}$
and set
\begin{equation*}
\hat{c}
=
\inf_{\eta\in\Gamma}
\max_{u\in Q}\mathcal{I}(\eta(u)).
\end{equation*}
Then $\hat{c}\geq a.$
Moreover, $\mathcal{I}$ possesses a Palais-Smale sequence at level $\hat{c}$. In particular, if $\mathcal{I}$ satisfies the $(PS)_{\hat{c}}$ condition, then $\hat{c}$ is a critical value of $\mathcal{I}$.
\end{theorem}
\begin{lemma}\label{lem:finite-dimensional-negativity}
Assume that \ref{f3} and \ref{f5} hold. If
$Z\subset H_0^1(\Omega)$ is finite-dimensional, then
\begin{equation*}
\mathcal{I}(u)\rightarrow-\infty
\quad\text{as }\Vert u\Vert_B\rightarrow+\infty,
\quad
u\in Z.
\end{equation*}
\end{lemma}
\begin{proof}
Let $Z\subset H_0^1(\Omega)$ be finite-dimensional.
Since all norms are equivalent on $Z$, and since $\Vert\cdot\Vert_{L^\theta(\Omega)}$ is a norm
on $Z$, there exists a constant $C_Z>0$ such that
\begin{equation}\label{eq:norm-equivalence-Z}
\Vert u\Vert_{L^\theta(\Omega)}
\geq
C_Z\Vert u\Vert_B,
\quad
u\in Z.
\end{equation}
Using Lemma \ref{eq:superquadratic-F-infty} and \eqref{eq:norm-equivalence-Z}, for every
$u\in Z$ we get
\begin{align*}
\mathcal{I}(u)
&=
\frac12\Vert u\Vert_B^2-\int_\Omega F(x,u)\,\mathrm{d}x\\
&\leq
\frac12\Vert u\Vert_B^2
-
C_{1,\theta}\int_\Omega \vert u\vert ^\theta\,\mathrm{d}x
+
C_{2,\theta}\vert \Omega\vert \\
&=
\frac12\Vert u\Vert_B^2
-
C_{1,\theta}\Vert u\Vert_{L^\theta(\Omega)}^\theta
+
C_{2,\theta} \vert \Omega\vert \\
&\leq
\frac12\Vert u\Vert_B^2
-
C_{1,\theta} C_Z^\theta\Vert u\Vert_B^\theta
+
C_{2,\theta}\vert \Omega\vert .
\end{align*}
Since $\theta>2$, the negative term of order $\Vert u\Vert_B^\theta$ dominates the
quadratic term. Therefore
\begin{equation*}
\mathcal{I}(u)\to-\infty
\quad\text{as }\Vert u\Vert_B\to+\infty,
\quad u\in Z.
\end{equation*}
The proof is complete.
\end{proof}

\begin{lemma}\label{linking-geometry}
Assume that \ref{f1}, \ref{f3}, \ref{f5}, \ref{f7}, \ref{f8} hold and $\delta \leq \lambda_{1} + \gamma$. Then the functional $\mathcal{I}$ satisfies the linking geometry. More precisely, there exist \(a>0\), \(\rho>0\), and \(n_0\in\mathbb N\)
such that, for every \(n\geq n_0\), one can find \(R_n>\rho\) and
\(z_n\in\mathcal W_k\), with $\Vert z_n \Vert_B=1$, satisfying the following
properties:
\begin{boldromanenum}
   \item
    $
    \mathcal I(u)\geq a,
    $ for $u\in\mathcal W_k$ with $\Vert u \Vert_B=\rho$;
    \item
    $
    \mathcal I(u)\leq 0,
    $ for 
    $u\in\partial Q_n$, where
    \[
    Q_n
    =
    \left\{
    v+tz_n:\ v\in\mathcal V_k,\ \Vert v \Vert_B\leq R_n,\ 0\leq t\leq R_n
    \right\}.
    \]
\end{boldromanenum}
\end{lemma}

\begin{proof}
We first prove the positive geometry on $\mathcal W_k$. For this, let $\varepsilon>0$, $q>2$, and $\alpha>\alpha_0$ be fixed. By \ref{f1} and \ref{f7} there exists $C_{1} >0$ such that
\begin{equation*}
F(x,s)
\leq
\frac{\mu_{*}+\varepsilon}{2}s^2
+
C_{1} \vert s \vert^q\exp(\alpha s^2),
\quad
(x,s)\in\overline{\Omega}\times\mathbb R.
\end{equation*}
From H\"{o}lder's inequality (for $p' = q' = 2$), $\Vert u \Vert_{B} = \rho$ and taking $\rho$ such that $2 \alpha \rho^{2} \leq 4 \pi,$ by Trudinger-Moser inequality Lemma \ref{eq:TMeditadaC} we have that 
\begin{align*}
    \int_{\Omega} F(x,u) \; \mathrm{d}x & \leq  \frac{\mu_{*}+\varepsilon}{2} \Vert u \Vert_{L^{2}(\Omega)}^{2} + C_{1} \Vert u \Vert^{q}_{L^{2q}(\Omega)} \left(\int_{\Omega} \exp(2 \alpha u^{2}) \; \mathrm{d}x\right)^{1/2} \\ &=  \frac{\mu_{*}+\varepsilon}{2} \Vert u \Vert_{L^{2}(\Omega)}^{2} + C_{1} \Vert u \Vert^{q}_{L^{2q}(\Omega)} \left(\int_{\Omega} \exp(2 \alpha \Vert u \Vert_{B}^{2} (u/ \Vert u \Vert_{B})^{2}) \; \mathrm{d}x\right)^{1/2} \\ & = \frac{\mu_{*}+\varepsilon}{2} \Vert u \Vert_{L^{2}(\Omega)}^{2} + C_{1} \Vert u \Vert^{q}_{L^{2q}(\Omega)} \left(\int_{\Omega} \exp(2 \alpha \rho^{2} (u/ \Vert u \Vert_{B})^{2}) \; \mathrm{d}x\right)^{1/2} \\ &\leq \frac{\mu_{*}+\varepsilon}{2} \Vert u \Vert_{L^{2}(\Omega)}^{2} + C_{2} \Vert u \Vert^{q}_{L^{2q}(\Omega)},
\end{align*}
where $C_{2} > 0$ is a positive constant.

Let $u\in\mathcal W_k$. Using the spectral inequality on $\mathcal W_k$ (Lemma \ref{eq:spectralinequalities} $(ii)$), Sobolev embedding and the previous estimate, we obtain
\begin{align*}
\mathcal{I}(u)
&=
\frac12\Vert u\Vert_B^2-\int_{\Omega}F(x,u)\, \mathrm{d}x\\
&\geq
\frac12\Vert u\Vert_B^2
-
\frac{\mu_{*}+\varepsilon}{2}\Vert u\Vert_{L^2(\Omega)}^2
-
C_{2} \Vert u \Vert^{q}_{L^{2q}(\Omega)}\\
&\geq
\frac12
\left(
1-\frac{\mu_{*}+\varepsilon}{\hat{\lambda}_{k+1}}
\right)
\Vert u\Vert_B^2
-
C_{2} \Vert u \Vert^{q}_{L^{2q}(\Omega)} \\ & \geq \frac12
\left(
1-\frac{\mu_{*}+\varepsilon}{\hat{\lambda}_{k+1}}
\right)
\Vert u\Vert_B^2
- C_{3} \Vert u \Vert^{q}_{B}.
\end{align*}
Since $\mu_{*}<\hat{\lambda}_{k+1}$, we choose $\varepsilon>0$ small enough so that
$
1-(\mu_{*}+\varepsilon)/\hat{\lambda}_{k+1}>0.
$
Therefore, item $(i)$ holds for $\rho > 0$ sufficiently small.

We now prove the nonpositive estimate on the finite-dimensional part. Let $G_n$ be the Moser function defined in \eqref{eq:mosersequenced}. Let \(P_k:H_0^1(\Omega)\to \mathcal{V}_k\) be the \(B\)-orthogonal projection onto
\(\mathcal{V}_k\). Define
\begin{equation*}
\widetilde z_n \coloneqq G_{n}-P_{k}G_{n}.
\end{equation*}
Then $\widetilde z_n\in \mathcal{W}_k$. 

Since $G_n\rightharpoonup 0$ in $H_0^1(\Omega)$, we claim that
$
P_kG_n\rightarrow 0,
$
in $H_0^1(\Omega)$. Indeed, since
$
\mathcal V_k=\operatorname{span}\{\varphi_1,\ldots,\varphi_k\},
$
and $(\varphi_j)$ diagonalizes the operator $\operatorname{T} = -\Delta+B$, we may write
\begin{equation*}
P_kG_n
=
\sum_{j=1}^{k}
\frac{\langle G_n,\varphi_j\rangle_B}
{\Vert \varphi_j\Vert_B^2}
\varphi_j.
\end{equation*}
For each fixed $j=1,\ldots,k$, the weak convergence $G_n\rightharpoonup0$ in $H_0^1(\Omega)$ gives
$
\langle G_n,\varphi_j\rangle_B\rightarrow0, 
$
as $n \rightarrow +\infty$. Therefore, for each $j=1,\ldots,k$ we have
\begin{equation*}
\frac{\langle G_n,\varphi_j\rangle_B}
{\Vert \varphi_j\Vert_B^2}
\rightarrow0,
\quad \mbox{as } n \rightarrow +\infty.
\end{equation*}
Since the sum is finite, it follows that
$
P_kG_n\rightarrow0
$
strongly in $(\mathcal V_k,\Vert\cdot\Vert_B)$.
Furthermore, since \(V_k\) is finite-dimensional, convergence in \((V_k,\Vert \cdot \Vert_B)\)
is equivalent to convergence in \(H_0^1(\Omega)\). We conclude that
$
P_kG_n\rightarrow0
$ strongly in $H_0^1(\Omega)$,
that is,
$
\Vert P_kG_n\Vert_B=o_n(1).
$
Therefore, using the asymptotic estimate \eqref{eq:normassintoticad} and the Cauchy–Schwarz inequality, we deduce
\begin{align*}
\Vert \widetilde z_n\Vert_B^2
&=
\Vert G_n-P_kG_n\Vert_B^2\\
&=
\Vert G_n\Vert_B^2
-
2\langle G_n,P_kG_n\rangle_B
+
\Vert P_kG_n\Vert_B^2\\
&=
1+o_n(1).
\end{align*}

For each $n \in \mathbb{N}$, define
$
z_n=\widetilde z_n/\Vert \widetilde z_n\Vert_B.
$
Then $z_n\in\mathcal W_k$ and $\Vert z_n\Vert_B=1$. It remains to prove that \(\mathcal{I}(u)\leq0\) for all \(u\in\partial Q_n\), provided
\(R_n>0\) is chosen sufficiently large. For this, let
$
Z_n=\mathcal V_k\oplus\mathbb R z_n.
$
Since $Z_n$ is finite-dimensional, for $u\in Z_n$, it follows from Lemma \ref{lem:finite-dimensional-negativity} that
\begin{equation*}
\mathcal{I}(u)\to-\infty,
\quad\text{as }\Vert u\Vert_B\to+\infty.
\end{equation*}
Hence, we may choose
$R_n>\rho$, possibly depending on $n \in \mathbb{N},$ such that
$
\mathcal{I}(u)\leq0,
$
for
$u\in Z_n$ with $\Vert u\Vert_B\geq R_{n}$.

We now prove that $\mathcal{I}(u) \leq0$ for each $u \in \partial Q_n$. The boundary of $Q_n$ relative to
$Z_n$ is contained in the union
$
\partial Q_n
\subset
Q_{n}^{0}\cup Q_{n}^{1}\cup Q_{n}^{2},
$
where
\begin{align*}
Q_{n}^{0}
&=
\left\{
v\in\mathcal V_k:\ \Vert v\Vert_B\leq R_{n}
\right\},\\
Q_{n}^{1}
&=
\left\{
v+tz_n:\ v\in\mathcal V_k,\ \Vert v\Vert_B=R_{n},\ 0\leq t\leq R_{n}
\right\},\\
Q_{n}^{2}
&=
\left\{
v+R_{n}z_n:\ v\in\mathcal V_k,\ \Vert v\Vert_B\leq R_{n}
\right\}.
\end{align*}
If $u\in Q_{n}^{0}$, then $u\in\mathcal V_k$. In view of \ref{f8} and the spectral
inequality on $\mathcal V_k$ (Lemma \ref{eq:spectralinequalities} $(i)$), we deduce
\begin{align*}
\mathcal{I}(u)
&=
\frac12\Vert u\Vert_B^2-\int_\Omega F(x,u)\,\mathrm{d}x\\
&\leq
\frac12\hat{\lambda}_k\Vert u\Vert_{L^2(\Omega)}^2
-
\frac12(\hat{\lambda}_k+\sigma)\Vert u\Vert_{L^2(\Omega)}^2\\
&=
-\frac{\sigma}{2}\Vert u\Vert_{L^2(\Omega)}^2
\leq0.
\end{align*}
Now, let $u \in Q_{n}^{1}$. Then $u=v+tz_n$, with
$v\in\mathcal V_k$, $\Vert v\Vert_B=R_{n}$, and $0\leq t\leq R_{n}$. Since
$z_n\in\mathcal W_k$, $\Vert z_n\Vert_B=1$, and
$\mathcal V_k$ is orthogonal to $\mathcal W_k$ with respect to
$\langle\cdot,\cdot\rangle_B$, we obtain
\begin{equation*}
\Vert u\Vert_B^2
=
\Vert v\Vert_B^2+t^2\Vert z_n\Vert_B^2
=
R_{n}^2+t^2
\geq R_{n}^2.
\end{equation*}
Thus $\Vert u\Vert_B\geq R_{n}$. Since $u\in Z_n$, the choice of $R_{n}$ gives
$
\mathcal{I}(u)\leq0.
$
If $u\in Q_{n}^{2}$, then $u=v+R_{n}z_n$, with
$v\in\mathcal V_k$ and $\Vert v\Vert_B\leq R_{n}$. Again, by the
$\langle\cdot,\cdot\rangle_B$-orthogonality between $\mathcal V_k$ and
$\mathcal W_k$, we have
\begin{equation*}
\Vert u\Vert_B^2
=
\Vert v\Vert_B^2+R_{n}^2\Vert z_n\Vert_B^2
=
\Vert v\Vert_B^2+R_{n}^2
\geq R_{n}^2.
\end{equation*}
Hence $\Vert u\Vert_B\geq R_{n}$. Since $u\in Z_n$, it follows once more from the choice of $R_{n}$
that
$
\mathcal{I}(u)\leq0.
$
Therefore,
$
\mathcal{I}(u)\leq0,
$
for
$
u\in\partial Q_n.
$
The proof is complete.
\end{proof}
\subsection{The critical linking level estimate}
We are now in a position to introduce the minimax level associated with the linking structure obtained above. For each sufficiently large \(n\in\mathbb N\), so that the conclusions of the previous lemma hold, let
\begin{equation*}
Q_n
=
\left\{
v+tz_n:\ v\in\mathcal V_k,\ \Vert v\Vert_B\leq R_{n},\ 0\leq t\leq R_{n}
\right\}.
\end{equation*}
Define
\begin{equation*}
\Gamma_n
=
\left\{
\eta\in C(Q_n,H_0^1(\Omega)):\ \eta(u)=u \ \text{for all } u\in\partial Q_n
\right\}.
\end{equation*}
The corresponding linking minimax level is given by
\begin{equation*}
\hat{c}_n
=
\inf_{\eta\in\Gamma_n}
\max_{u\in Q_n}\mathcal{I}(\eta(u)).
\end{equation*}
In view of Lemma \ref{linking-geometry} and Linking Theorem \ref{linkingtheorem}, we have
$
\hat{c}_n\geq a>0.
$
Moreover, there exists a Palais-Smale sequence for $\mathcal{I}$ at the level $\hat{c}_n$.

The concentration estimate must be compatible with the spectral
decomposition used in the linking construction. Since the concentrating
direction is required to belong to $\mathcal{W}_k$, the Moser function $G_n$
cannot be used directly, and we therefore remove its $\mathcal{V}_k$-component
by means of the spectral projection $P_k$.
\begin{lemma}[Projected Moser estimate in the $B$-norm] \label{lem:projected-moser-B} Suppose that $\delta \leq \lambda_{1} + \gamma$. 
Let $G_n$ be the Moser function defined in \eqref{eq:mosersequenced}, and let $P_k$ denote the
$B$-orthogonal projection onto $\mathcal V_k$ with respect to the inner product
$\langle\cdot,\cdot\rangle_B$. For \(n\in\mathbb{N}\) sufficiently large, set
$
\widetilde z_n
=
G_n-P_kG_n
$
and
$
z_n
=
\widetilde z_n/\Vert \widetilde z_n\Vert_B.
$
Then $z_n\in\mathcal W_k$, $\Vert z_n\Vert_B=1$, and
$
z_n\rightharpoonup0,
$
weakly in $H_0^1(\Omega)$. Moreover, there exists a constant $C_{B,0}>0$ such that
\begin{equation}\label{eq:projected-moser-pointwise}
z_n^2(x)
\geq
\frac{\log n}{2\pi}
-
\frac{C_{B,0}}{4\pi^2}
+
o_n(1),
\quad
x\in B(x_0,d/n).
\end{equation}
Consequently, if
$
t_n^2\geq 4\pi/\alpha_0
$
for $n\in\mathbb{N}$ sufficiently large, then
\[
\liminf_{n\to+\infty}
\int_{B(x_0,d/n)}
\exp(\alpha_0t_n^2z_n^2)\,\mathrm{d}x
\geq
\pi d^2\exp\left(-\frac{C_{B,0}}{\pi}\right).
\]
\end{lemma}

\begin{proof}
Since the eigenfunctions $\{\varphi_j\}_{j\in\mathbb N}$ are
orthonormal in $L^2(\Omega)$ and diagonalize the operator
$\operatorname{T} = -\Delta+B$, we have
$
\langle \varphi_i,\varphi_j\rangle_B
=
\hat{\lambda}_i\delta_{ij},
$
for $i,j\in\mathbb N$. Consequently, for every $u\in H_0^1(\Omega)$, its
$B$-orthogonal projection onto $\mathcal V_k$ is given by
\begin{equation*}
P_ku
=
\sum_{j=1}^k
\frac{\langle u,\varphi_j\rangle_B}
{\Vert \varphi_j\Vert_B^2}\varphi_j.
\end{equation*}
Since
$
\Vert \varphi_j\Vert_B^2=\hat{\lambda}_j
$
and
\begin{equation*}
\langle u,\varphi_j\rangle_B
=
\hat{\lambda}_j
\int_\Omega u\varphi_j\,\mathrm{d}x,
\end{equation*}
we obtain
\begin{equation} \label{eq:identidadedaprojecao2307}
P_ku
=
\sum_{j=1}^k
\left(
\int_\Omega u\varphi_j\,\mathrm{d}x
\right)\varphi_j.
\end{equation}
Thus, the $B$-orthogonal projection onto $\mathcal V_k$
coincides with the usual $L^2(\Omega)$-orthogonal spectral
projection onto
$
\mathcal V_k
=
\operatorname{span}\{\varphi_1,\ldots,\varphi_k\}.
$
Applying the identity \eqref{eq:identidadedaprojecao2307} to $G_n$, we have
\begin{equation*}
P_kG_n
=
\sum_{j=1}^k
\left(
\int_\Omega G_n\varphi_j\,\mathrm{d}x
\right)\varphi_j.
\end{equation*}
Since $\varphi_j\in L^\infty(\Omega)$ for every
$j=1,\ldots,k$, Hölder's inequality gives
\begin{align*}
\Vert P_kG_n\Vert_{L^\infty(\Omega)}
&\leq
\sum_{j=1}^k
\left\vert 
\int_\Omega G_n\varphi_j\,\mathrm{d}x
\right\vert 
\Vert \varphi_j\Vert_{L^\infty(\Omega)}
\\
&\leq
\left(
\sum_{j=1}^k
\Vert \varphi_j\Vert_{L^\infty(\Omega)}
\right)
\Vert G_n\Vert_{L^2(\Omega)}.
\end{align*}
The estimates for the Moser sequence imply that
\begin{equation*}
\Vert G_n\Vert_{L^2(\Omega)}^2
=
O\left(\frac{1}{\log n}\right).
\end{equation*}
Therefore, there exists $C_{1,k}>0$ such that
\begin{equation}
\Vert P_kG_n\Vert_{L^\infty(\Omega)}
\leq
\frac{C_{1,k}}{\sqrt{\log n}},
\label{eq:projection-Moser-Linfty}
\end{equation}
for $n\in\mathbb{N}$ sufficiently large.

We next estimate the $B$-norm of $P_kG_n$. Since
$P_kG_n\in\mathcal V_k$, the spectral inequality on
$\mathcal V_k$ (Lemma \ref{eq:spectralinequalities} $(i)$) yields
\begin{equation*}
\Vert P_kG_n\Vert_B^2
\leq
\hat{\lambda}_k
\Vert P_kG_n\Vert_{L^2(\Omega)}^2.
\end{equation*}
Moreover, because $P_k$ is also the $L^2(\Omega)$-orthogonal
projection onto $\mathcal V_k$, we have
\begin{equation*}
\Vert P_kG_n\Vert_{L^2(\Omega)}
\leq
\Vert G_n\Vert_{L^2(\Omega)}.
\end{equation*}
Consequently, there exists $C_{2,k}>0$ such that
\begin{equation}
\Vert P_kG_n\Vert_B^2
\leq
\frac{C_{2,k}}{\log n}.
\label{eq:projection-Moser-Bnorm}
\end{equation}

By the definition
$
\widetilde z_n=G_n-P_kG_n,
$
we have $\widetilde z_n\in\mathcal W_k$ and
$
\langle \widetilde z_n,P_kG_n\rangle_B=0.
$
Hence, by the Pythagorean identity,
\begin{equation*}
\Vert G_n\Vert_B^2
=
\Vert \widetilde z_n\Vert_B^2
+
\Vert P_kG_n\Vert_B^2.
\end{equation*}
Using
\begin{equation*}
\Vert G_n\Vert_B^2
=
1+O\left(\frac{1}{\log n}\right)
\end{equation*}
together with \eqref{eq:projection-Moser-Bnorm}, we obtain
\begin{equation*}
\Vert \widetilde z_n\Vert_B^2
=
1+O\left(\frac{1}{\log n}\right).
\end{equation*}
In particular,
$
\Vert \widetilde z_n\Vert_B\rightarrow1,
$
as $n \rightarrow +\infty$. Thus, $z_n$ is well defined for all sufficiently large $n$,
belongs to $\mathcal W_k$, and satisfies
$
\Vert z_n\Vert_B=1.
$
Since $G_n\rightharpoonup0$ in $H_0^1(\Omega)$ and
$\mathcal V_k$ is finite-dimensional, we have
$
P_kG_n\rightarrow0,
$
strongly in $H_0^1(\Omega)$. Therefore,
$
\widetilde z_n
=
G_n-P_kG_n
\rightharpoonup0,
$
weakly in $H_0^1(\Omega)$. Since $\Vert \widetilde z_n\Vert_B\to1$, it follows that
$
z_n
=
\widetilde z_n/\Vert \widetilde z_n\Vert_B\rightharpoonup0,
$
weakly in $H_0^1(\Omega)$.

It remains to establish the pointwise estimate on the
concentration ball. Recall that
\begin{equation*}
G_n(x)
=
\frac{\sqrt{\log n}}{\sqrt{2\pi}}, \quad x\in B(x_0,d/n).
\end{equation*}
In view of \eqref{eq:projection-Moser-Linfty}, we have
\begin{equation*}
\widetilde z_n(x)
=
G_n(x)-P_kG_n(x)
\geq
\frac{\sqrt{\log n}}{\sqrt{2\pi}}
-
\frac{C_{1,k}}{\sqrt{\log n}}.
\end{equation*}
On the other hand, there exists $C_{3,k}>0$ such that
\begin{equation*}
\Vert \widetilde z_n\Vert_B
\leq
\left(
1+\frac{C_{3,k}}{\log n}
\right)^{1/2}.
\end{equation*}
Therefore, uniformly for $x\in B(x_0,d/n)$,
\begin{equation*}
z_n(x)
=
\frac{\widetilde z_n(x)}
{\Vert \widetilde z_n\Vert_B}
\geq
\frac{
\dfrac{\sqrt{\log n}}{\sqrt{2\pi}}
-
\dfrac{C_{1,k}}{\sqrt{\log n}}
}{
\left(
1+\dfrac{C_{3,k}}{\log n}
\right)^{1/2}
}.
\end{equation*}
Using
\begin{equation*}
\left(
1+\frac{C_{3,k}}{\log n}
\right)^{-1/2}
=
1-\frac{C_{3,k}}{2\log n}
+
O\left(\frac{1}{(\log n)^2}\right),
\end{equation*}
we find a constant $C_{4,k}>0$ such that
\begin{equation*}
z_n(x)
\geq
\frac{\sqrt{\log n}}{\sqrt{2\pi}}
-
\frac{C_{4,k}}{\sqrt{\log n}}
+
o\left(\frac{1}{\sqrt{\log n}}\right),
\end{equation*}
uniformly for $x\in B(x_0,d/n)$. For \(n\in\mathbb{N}\) large, the right-hand side is positive. Hence, squaring the
preceding inequality and enlarging the constant if necessary, we obtain, for some constant $C_{B,0} > 0$,
\begin{equation}
z_n^2(x)
\geq
\frac{\log n}{2\pi}
-
\frac{C_{B,0}}{4\pi^2}
+
o_n(1),
\label{eq:projected-Moser-pointwise}
\end{equation}
uniformly for \(x\in B(x_0,d/n)\).

Finally, suppose that
$
t_n^2\geq 4\pi/\alpha_0,
$
for $n\in\mathbb{N}$ sufficiently large. From
\eqref{eq:projected-Moser-pointwise}, we obtain, uniformly
for $x\in B(x_0,d/n)$,
\begin{equation*}
\alpha_0t_n^2z_n^2(x)
\geq
4\pi z_n^2(x)
\geq
2\log n-\frac{C_{B,0}}{\pi}+o_n(1).
\end{equation*}
It follows that
\begin{align*}
\int_{B(x_0,d/n)}
\exp(\alpha_0t_n^2z_n^2)\,\mathrm{d}x
&\geq
\exp\left(
2\log n-\frac{C_{B,0}}{\pi}+o_n(1)
\right)
\left\vert B(x_0,d/n)\right\vert 
\\
&=
\pi d^2
\exp\left(
-\frac{C_{B,0}}{\pi}+o_n(1)
\right).
\end{align*}
Taking the lower limit as $n\to+\infty$, we conclude that
\begin{equation*}
\liminf_{n\to+\infty}
\int_{B(x_0,d/n)}
\exp(\alpha_0t_n^2z_n^2)\,\mathrm{d}x
\geq
\pi d^2
\exp\left(-\frac{C_{B,0}}{\pi}\right).
\end{equation*}
This completes the proof.
\end{proof}

\begin{lemma}\label{lem:bounded-maximizing-parameters}
Assume that \ref{f1}, \ref{f3}-\ref{f5}, \ref{f7}-\ref{f8} hold and $\delta \leq \lambda_{1} + \gamma$. Suppose that
\begin{equation*}
\max_{u\in Q_n}\mathcal{I}(u)
\geq
\frac{2\pi}{\alpha_0},
\end{equation*}
for $n\in\mathbb{N}$ sufficiently large. For each such $n\in\mathbb{N}$, let
\begin{equation*}
u_n=v_n+t_nz_n\in Q_n,
\quad
v_n\in\mathcal V_k,\quad 0\leq t_n\leq R_n,
\end{equation*}
be such that
\begin{equation*}
\mathcal{I}(u_n)=\max_{u\in Q_n}\mathcal{I}(u).
\end{equation*}
Then $(v_n)_{n\in\mathbb N}$ is bounded in $\mathcal V_k$ and
$(t_n)_{n\in\mathbb N}$ is bounded in $\mathbb R$.
\end{lemma}

\begin{proof}
Since $\mathcal{I}\leq0$ on $\partial Q_n$ and
\begin{equation*}
\mathcal{I}(u_n)\geq\frac{2\pi}{\alpha_0}>0,
\end{equation*}
the point $u_n$ belongs to the relative interior of $Q_n$ in $Z_{n}$. Therefore,
\begin{equation*}
\mathcal{I}'(u_n)\varphi=0,
\quad
\forall\,\varphi\in Z_n
\coloneqq 
\mathcal V_k\oplus\mathbb Rz_n.
\end{equation*}
Taking $\varphi=u_n$, we obtain
\begin{equation}\label{eq:critical-identity-maximizers}
\Vert u_n\Vert_B^2
=
\int_\Omega f(x,u_n)u_n\,\mathrm{d}x.
\end{equation}
Since $v_n\in\mathcal V_k$, $z_n\in\mathcal W_k$ and
$\Vert z_n\Vert_B=1$, the $B$-orthogonality between \(\mathcal{V}_k\) and \(\mathcal{W}_k\) gives
\begin{equation*}
\Vert u_n\Vert_B^2
=
\Vert v_n\Vert_B^2+t_n^2.
\end{equation*}
Moreover, since the $B$-orthogonal spectral decomposition also coincides with
the $L^2(\Omega)$-orthogonal decomposition, we have
\begin{equation*}
\Vert u_n\Vert_{L^2(\Omega)}^2
=
\Vert v_n\Vert_{L^2(\Omega)}^2
+
t_n^2\Vert z_n\Vert_{L^2(\Omega)}^2.
\end{equation*}
Using \ref{f8} and the spectral inequality on $\mathcal V_k$ (Lemma \ref{eq:spectralinequalities} $(i)$), we obtain
\begin{align*}
\mathcal{I}(u_n)
&\leq
\frac12\Vert v_n\Vert_B^2
+
\frac12t_n^2
-
\frac{\hat{\lambda}_k+\sigma}{2}
\left(
\Vert v_n\Vert_{L^2(\Omega)}^2
+
t_n^2\Vert z_n\Vert_{L^2(\Omega)}^2
\right)\\
&\leq
\frac12t_n^2
-
\frac{\sigma}{2}\Vert v_n\Vert_{L^2(\Omega)}^2
-
\frac{\hat{\lambda}_k+\sigma}{2}
t_n^2\Vert z_n\Vert_{L^2(\Omega)}^2.
\end{align*}
Since \(\mathcal{I}(u_n)\geq2\pi/\alpha_0\), it follows that
\begin{equation}\label{eq:tn-lower-control-vn}
t_n^2
\geq
\frac{4\pi}{\alpha_0}
+
\sigma\Vert v_n\Vert_{L^2(\Omega)}^2
+
(\hat{\lambda}_k+\sigma)
t_n^2\Vert z_n\Vert_{L^2(\Omega)}^2.
\end{equation}
In particular,
\begin{equation}\label{eq:tn-away-zero}
t_n^2\geq\frac{4\pi}{\alpha_0}.
\end{equation}

Set
$
r_n\coloneqq \Vert v_n\Vert_B.
$
If $r_n=0$ for infinitely many $n$, we pass to that subsequence and the
argument below for the first case applies directly. We may therefore assume
that $r_n>0$. Passing to a subsequence, one of the following alternatives
holds:
\begin{itemize}
    \item[$(i)$] $t_n/r_n\geq c_0>0$;\vspace{0.1cm}
    \item[$(ii)$] $t_n/r_n\rightarrow0,$ as $n \rightarrow +\infty$.
\end{itemize}

If $(i)$ holds, then $r_n\leq t_n/c_0$, and consequently,
\begin{equation}\label{eq:un-controlled-by-tn}
\Vert u_n\Vert_B^2
=
r_n^2+t_n^2
\leq
C_0t_n^2,
\end{equation}
for some constant $C_0>0$ independent of $n$. Since $\mathcal V_k$ is finite-dimensional, there exists $C_1>0$ such that
\begin{equation*}
\Vert v\Vert_{L^\infty(\Omega)}
\leq
C_1\Vert v\Vert_B,
\quad
\forall\,v\in\mathcal V_k.
\end{equation*}
Thus,
\begin{equation*}
\frac{\Vert v_n\Vert_{L^\infty(\Omega)}}{t_n}
\leq
\frac{C_1}{c_0}.
\end{equation*}
On the other hand, the pointwise estimate for the projected Moser functions implies that
\[
\inf_{x\in B(x_0,d/n)} z_n(x)\longrightarrow+\infty.
\]
Consequently,
\[
\inf_{x\in B(x_0,d/n)}
\left(
z_n(x)-\frac{|v_n(x)|}{t_n}
\right)
\longrightarrow+\infty.
\]
Hence, for every fixed $\varepsilon\in(0,1)$ and for all sufficiently large $n$
\begin{equation*}
u_n(x)
=
v_n(x)+t_nz_n(x)
\geq
(1-\varepsilon)t_nz_n(x)>0,
\quad
x\in B(x_0,d/n).
\end{equation*}
Fix \(M>0\). By \ref{f4}, there exists \(s_M>0\) such that
\[
f(x,s)s\geq M \exp({\alpha_0s^2}),
\quad x\in\overline\Omega,\ s\geq s_M.
\]
Since $\inf \limits_{x\in B(x_0,d/n)}u_n(x)\longrightarrow+\infty,$ for all sufficiently
large \(n\) we have
\[
f(x,u_n)u_n\geq M \exp({\alpha_0u_n^2})
\quad\text{on }B(x_0,d/n).
\]

It follows from
\eqref{eq:critical-identity-maximizers},
\eqref{eq:un-controlled-by-tn} and the pointwise estimate for $z_n$ that there
exist constants $A_1,A_2,A_3>0$, independent of $n$, such that
\begin{equation}\label{eq:exponential-control-tn}
A_1t_n^2
\geq
A_2
\exp\left\{
\left[
\frac{\alpha_0(1-\varepsilon)^2t_n^2}{2\pi}-2
\right]\log n
-
A_3t_n^2
\right\}.
\end{equation}
The preceding inequality implies that $(t_n)_{n \in \mathbb{N}}$ is bounded. Indeed, if, up to a
subsequence, $t_n\to+\infty$, then the right-hand side of
\eqref{eq:exponential-control-tn} grows faster than the quadratic term on the
left-hand side, which is impossible. Since
$
r_n\leq t_n/c_0,
$
the sequence $(v_n)_{n \in \mathbb{N}}$ is also bounded.

It remains to exclude the second alternative. If $(ii)$ holds, then in view of \eqref{eq:tn-away-zero}, we have that
$
r_n\rightarrow+\infty, 
$
as $n \rightarrow +\infty$.
Define
$
w_n \coloneqq v_n/r_n.
$
Since $\mathcal V_k$ is finite-dimensional, up to a subsequence,
$
w_n\rightarrow w,
$
strongly in $\mathcal V_k$, with $\Vert w\Vert_B=1$.
Moreover, since \(G_n\to0\) a.e. in \(\Omega\), \(P_kG_n\to0\) uniformly in
\(\Omega\), and \(\Vert \widetilde z_n \Vert_B\to1\), we have
\[
z_n=\frac{G_n-P_kG_n}{\Vert \widetilde z_n \Vert_B}\to0
\quad\text{a.e. in }\Omega.
\] 
Since $z_n\to0$ almost everywhere in $\Omega$,
\begin{equation*}
\frac{u_n}{r_n}
=
w_n+\frac{t_n}{r_n}z_n
\rightarrow w,
\quad\text{a.e. in }\Omega.
\end{equation*}
Fix $\alpha\in(0,\alpha_0)$. By \ref{f1} and \ref{f5}, there exists $S>0$
such that
\begin{equation*}
f(x,s)s
\geq
\exp(\alpha s^2),
\quad
x\in\overline{\Omega},\quad \vert s\vert \geq S.
\end{equation*}
Since $t_n/r_n\to0$, we also have
\begin{equation*}
\Vert u_n\Vert_B^2
=
r_n^2+t_n^2
\leq
2r_n^2,
\end{equation*}
for $n\in\mathbb{N}$ sufficiently large. Therefore,
\eqref{eq:critical-identity-maximizers} yields
\begin{equation*}
1
\geq
\frac{1}{2r_n^2}
\int_{\{\vert u_n\vert \geq S\}}
\exp(\alpha u_n^2)\,\mathrm{d}x.
\end{equation*}
Since $w\neq 0$, the set
\[
\{x\in\Omega:\ \vert w(x) \vert >0\}
\]
has positive measure. Moreover,
\[
\{x\in\Omega:\ \vert w(x) \vert >0\}
=
\bigcup_{m=1}^{+\infty}
\left\{x\in\Omega:\ \vert w(x) \vert \geq\frac{1}{m}\right\},
\]
there exists $m_0\in\mathbb{N}$ such that the measurable set
\[
E \coloneqq \left\{x\in\Omega:\ \vert w(x) \vert \geq\frac{1}{m_0}\right\}
\]
has positive measure. Setting $\delta_{0} \coloneqq 1/m_0$, we have
\[
|w(x)|\geq\delta_{0}>0,\qquad x\in E.
\]
Hence, for almost every $x\in E$, we have
\begin{equation*}
\vert u_n(x)\vert 
\geq
\frac{\delta_{0}}{2}r_n
\rightarrow+\infty, \quad \mbox{as } n \rightarrow +\infty.
\end{equation*}
Consequently,
\begin{equation*}
\frac{1}{2r_n^2}\exp({\alpha u_n(x)^2})\chi_{\{\vert u_n \vert \ge S\}}(x)
\to+\infty
\quad\text{for a.e. }x\in E.    
\end{equation*}
Fatou's lemma then implies
\begin{equation*}
\frac{1}{2r_n^2}
\int_{\{\vert u_n \vert \ge S\}}\exp({\alpha u_n^2})\, \mathrm{d} x \to+\infty,    
\end{equation*}
which contradicts the preceding estimate.

Hence, the second alternative cannot occur. Therefore both $(t_n)_{n \in \mathbb{N}}$ and
$(v_n)_{n \in \mathbb{N}}$ are bounded.
\end{proof}
\begin{remark}\label{rmk:CB-beta}
\emph{(On the constants $C_B$ and $\beta$.)}
Let $B(x_0,d)\subset\Omega$ be the ball fixed in \eqref{eq:mosersequenced}, let
$C_{B,0}>0$ be the constant provided by Lemma~\ref{lem:projected-moser-B}, and let
$K_k>0$ be such that
\[
\Vert v \Vert_{L^{\infty}(\Omega)}\leq K_k \Vert v \Vert_{L^{2}(\Omega)},
\qquad v\in V_k,
\]
which exists because $\dim V_k<+\infty$. Set
\begin{equation}\label{eq:def-CB}
C_B:=C_{B,0}+\frac{8\pi^{2}K_k^{2}}{\sigma}.
\end{equation}

\begin{enumerate}
\item[(i)] The constant $C_B$ is determined by the linking construction and by the
assumptions already in force. Indeed, $C_{B,0}$ and $K_k$ depend only on $k$, on
$\Omega$, on the eigenfunctions $\varphi_1,\dots,\varphi_k$ and on the choice of
$B(x_0,d)$, while $\sigma$ is the constant appearing in \ref{f8}. In particular, $C_B$
does not depend on the parameter $\beta$ introduced below.

\item[(ii)] Assumption \ref{f4} enters the argument only through the following
consequence: for every $\beta>0$ there exists $s_\beta>0$ such that
\begin{equation}\label{eq:f4-beta}
f(x,s)s\geq\beta\exp(\alpha_0 s^{2}),
\qquad x\in\Omega,\ s\geq s_\beta.
\end{equation}
Since \ref{f4} is stated as a limit equal to $+\infty$, the constant $\beta$ in
\eqref{eq:f4-beta} may be taken arbitrarily large. Accordingly, throughout the proof
of Lemma~\ref{eq:lemadonivelsaddle} we fix once and for all
\begin{equation}\label{eq:choice-beta}
\beta>\frac{4\,\exp({C_B/\pi})}{\alpha_0 d^{2}},
\end{equation}
together with the corresponding $s_\beta$. This choice is legitimate by~(i): the
right-hand side of \eqref{eq:choice-beta} is already determined at the moment
$\beta$ is chosen.

\item[(iii)] No assumption of the form $sf(x,s)\exp(-\alpha_0 s^{2})\geq\beta_0$
for some prescribed $\beta_0$ is imposed on $f$. Such an assumption would be
unsatisfactory in the present setting, since the admissible threshold involves
$C_B$, a constant produced by the projected Moser estimate of
Lemma~\ref{lem:projected-moser-B} and by the finite-dimensional component of the
linking set, and therefore not available at the level of the statement of the
hypotheses. Requiring the full limit \ref{f4} removes this dependence. 
\end{enumerate}
\end{remark}

\begin{lemma} \label{eq:lemadonivelsaddle}
Assume that \ref{f1}, \ref{f3}-\ref{f5}, \ref{f7}-\ref{f9} hold and $\delta \leq \lambda_{1} + \gamma$. Then, for some $n\in\mathbb{N}$, one has
\begin{equation*}
\max_{u\in Q_n} \mathcal{I}(u)
<
\frac{2\pi}{\alpha_0}.
\end{equation*}
Consequently,
\begin{equation*}
0<\hat{c}_n<\frac{2\pi}{\alpha_0}.
\end{equation*}
\end{lemma}

\begin{proof}
Recall the notation and constants introduced in Remark \ref{rmk:CB-beta}. Set
\[
T^2\coloneqq \frac{4\pi}{\alpha_0}.
\]
Suppose, by contradiction, that
\begin{equation}
\max_{u\in Q_n}\mathcal{I}(u)\geq\frac{2\pi}{\alpha_0}
=\frac{T^2}{2},
\label{eq:critical-linking-contradiction}
\end{equation}
for every sufficiently large $n\in\mathbb{N}$. Since $Q_n$ is compact in the finite-dimensional space
$
Z_n\coloneqq \mathcal V_k\oplus\mathbb Rz_n,
$
there exists
$
u_n=v_n+t_nz_n\in Q_n,
$
with
$
v_n\in\mathcal V_k
$
and
$
0\leq t_n\leq R_n,
$
such that
\begin{equation*}
\mathcal{I}(u_n)=\max_{u\in Q_n}\mathcal{I}(u).
\end{equation*}
By the linking geometry, $\mathcal{I}\leq0$ on $\partial Q_n$.
Therefore, in view of
\eqref{eq:critical-linking-contradiction}, $u_n$ belongs to
the relative interior of $Q_n$. Hence,
\begin{equation*}
\mathcal{I}'(u_n)\varphi=0,
\quad \forall\,\varphi\in Z_n.
\end{equation*}
Taking $\varphi=u_n$, we obtain
\begin{equation}
\Vert u_n\Vert_B^2
=
\int_\Omega f(x,u_n)u_n\,\mathrm{d}x.
\label{eq:maximizer-Nehari-identity}
\end{equation}
Since $v_n\in\mathcal V_k$, $z_n\in\mathcal W_k$, and
$\Vert z_n\Vert_B=1$, the $B$-orthogonality gives
\begin{equation*}
\Vert u_n\Vert_B^2=\Vert v_n\Vert_B^2+t_n^2.
\end{equation*}
Moreover, because the $B$-orthogonal spectral decomposition
also coincides with the $L^2(\Omega)$-orthogonal
decomposition,
\begin{equation*}
\Vert u_n\Vert_{L^{2}(\Omega)}^{2}
=
\Vert v_n\Vert_{L^{2}(\Omega)}^2+t_n^2\Vert z_n\Vert_{L^{2}(\Omega)}^2.
\end{equation*}
By \ref{f8} and the spectral inequality on
$\mathcal V_k$ (Lemma \ref{eq:spectralinequalities} $(i)$), we have
\begin{align*}
\mathcal{I}(u_n)
&\leq
\frac12\bigl(\Vert v_n\Vert_B^2+t_n^2\bigr)
-\frac{\hat\lambda_k+\sigma}{2}
\left(
\Vert v_n\Vert_{L^{2}(\Omega)}^2+t_n^2\Vert z_n\Vert_{L^{2}(\Omega)}^2
\right)
\\
&\leq
\frac{t_n^2}{2}
-\frac{\sigma}{2}\Vert v_n\Vert_{L^{2}(\Omega)}^2
-\frac{\hat\lambda_k+\sigma}{2}
t_n^2\Vert z_n\Vert_{L^{2}(\Omega)}^2.
\end{align*}
Combining this inequality with
\eqref{eq:critical-linking-contradiction}, we obtain
\begin{equation}
t_n^2
\geq
T^2+\sigma\Vert v_n\Vert_{L^{2}(\Omega)}^2
+(\hat\lambda_k+\sigma)t_n^2\Vert z_n\Vert_{L^{2}(\Omega)}^2.
\label{eq:critical-energy-gap}
\end{equation}
In particular,
\begin{equation}
t_n^2\geq T^2=\frac{4\pi}{\alpha_0}.
\label{eq:tn-critical-lower-bound}
\end{equation}
By Lemma \ref{lem:bounded-maximizing-parameters}, the sequences $(t_n)_{n \in \mathbb{N}}$ and $(v_n)_{n \in \mathbb{N}}$
are bounded. Thus, up to a subsequence,
$
t_n\rightarrow t_0\geq0
$
and
$v_n\rightarrow v_0$ strongly in $\mathcal V_k$, as $n \rightarrow +\infty$. Since $z_n\rightharpoonup0$ weakly in $H_0^1(\Omega)$ and the
embedding
$
H_0^1(\Omega)\hookrightarrow L^2(\Omega)
$
is compact, we also have
$
\Vert z_n\Vert_{L^{2}(\Omega)}\rightarrow0,
$
as $n \rightarrow +\infty$. Passing to the limit in \eqref{eq:critical-energy-gap}, we
obtain
\begin{equation}
t_0^2\geq T^2+\sigma\Vert v_0\Vert_{L^{2}(\Omega)}^2.
\label{eq:limit-critical-gap}
\end{equation}

We claim that
\begin{equation}
t_0^2=T^2
\quad\text{and}\quad
v_0=0.
\label{eq:critical-limit-identification}
\end{equation}
Suppose first that $t_0^2>T^2$. Choose $\vartheta \in(0,1)$ such
that
$
\alpha_0\vartheta^2t_0^2>4\pi.
$
Since $(v_n)_{n \in \mathbb{N}}$ is bounded in the finite-dimensional space
$\mathcal V_k$, it is bounded in $L^\infty(\Omega)$.
On the other hand,
\[
\inf_{x\in B(x_0,d/n)} z_n(x)\to+\infty.
\] 
Hence, for all sufficiently
large $n\in\mathbb{N}$,
\begin{equation*}
u_n(x)=v_n(x)+t_nz_n(x)
\geq \vartheta t_nz_n(x)>0,
\quad x\in B(x_0,d/n).
\end{equation*}
Fix $\beta > \frac{4\, \exp({C_B/\pi})}{\alpha_0 d^{2}}$, as in \eqref{eq:choice-beta}. By \ref{f4}, there exists
$s_\beta>0$ such that
\begin{equation}
f(x,s)s
\geq
\beta\exp(\alpha_0s^2),
\quad
x\in\overline\Omega,\quad s\geq s_\beta.
\label{eq:critical-asymptotic-lower-bound}
\end{equation}
Since
\[
\inf_{x\in B(x_0,d/n)} u_n(x)\to+\infty,
\] estimate
\eqref{eq:critical-asymptotic-lower-bound} applies there for
all sufficiently large $n$. Furthermore, by \ref{f5},
$f(x,s)s\geq0$ for every $s\in\mathbb R$. Consequently,
\begin{align*}
\int_\Omega f(x,u_n)u_n\,\mathrm{d}x
&\geq
\beta
\int_{B(x_0,d/n)}
\exp(\alpha_0u_n^2)\,\mathrm{d}x
\\
&\geq
\beta
\int_{B(x_0,d/n)}
\exp(\alpha_0\vartheta^2t_n^2z_n^2)\,\mathrm{d}x.
\end{align*}
Since $\alpha_0\vartheta^2t_n^2>4\pi$ for $n\in\mathbb{N}$ large, the pointwise estimate from Lemma \ref{lem:projected-moser-B}
implies that
\begin{equation*}
\int_{B(x_0,d/n)}
\exp(\alpha_0\vartheta^2t_n^2z_n^2)\,\mathrm{d}x \rightarrow + \infty, \quad \mbox{as } n \rightarrow +\infty.
\end{equation*}
Therefore,
\begin{equation*}
\int_\Omega f(x,u_n)u_n\,\mathrm{d}x\rightarrow+\infty, \quad \mbox{as } n \rightarrow +\infty.
\end{equation*}
This contradicts \eqref{eq:maximizer-Nehari-identity}, since
$(u_n)_{n \in \mathbb{N}}$ is bounded in $H_0^1(\Omega)$. Hence,
$
t_0^2\leq T^2.
$
Together with \eqref{eq:limit-critical-gap}, this gives
$
t_0^2=T^2
$
and
$
\Vert v_0\Vert_{L^{2}(\Omega)}=0.
$
Therefore $v_0=0$, proving
\eqref{eq:critical-limit-identification}.

We now establish the sharp concentration estimate for
$u_n=v_n+t_nz_n$. Since $\mathcal V_k$ is finite-dimensional,
there exists $K_k>0$ such that
\begin{equation}
\Vert v\Vert_{L^\infty(\Omega)}
\leq K_k\Vert v\Vert_{L^{2}(\Omega)},
\quad v\in\mathcal V_k.
\label{eq:Vk-L2-Linfty}
\end{equation}
From \eqref{eq:critical-energy-gap}, we have
\begin{equation}
t_n^2-T^2\geq\sigma\Vert v_n\Vert_{L^{2}(\Omega)}^2.
\label{eq:tn-vn-gap}
\end{equation}
Therefore, for every $x\in B(x_0,d/n)$,
\begin{align*}
u_n^2(x)
&=
\bigl(v_n(x)+t_nz_n(x)\bigr)^2
\\
&=
T^2z_n^2(x)
+\bigl(t_n^2-T^2\bigr)z_n^2(x)
+2t_nz_n(x)v_n(x)
+v_n^2(x)
\\
&\geq
T^2z_n^2(x)
+\sigma\Vert v_n\Vert_{L^{2}(\Omega)}^2z_n^2(x)
-2t_nK_k\Vert v_n\Vert_{L^{2}(\Omega)}\vert z_n(x)\vert .
\end{align*}
Completing the square, we obtain
\begin{equation*}
\sigma\Vert v_n\Vert_{L^{2}(\Omega)}^2z_n^2(x)
-2t_nK_k\Vert v_n\Vert_{L^{2}(\Omega)}\vert z_n(x)\vert 
\geq
-\frac{t_n^2K_k^2}{\sigma},
\end{equation*}
which implies that
\begin{equation}
u_n^2(x)
\geq
T^2z_n^2(x)
-\frac{t_n^2K_k^2}{\sigma}.
\label{eq:un-projected-Moser-lower-bound}
\end{equation}
Since $t_n^2\to T^2$, we have $t_n^2\leq2T^2$ for $n\in\mathbb{N}$ large. Consequently,
\begin{equation}
\alpha_0u_n^2(x)
\geq
4\pi z_n^2(x)-\frac{8\pi K_k^2}{\sigma}.
\label{eq:un-exponential-lower-bound}
\end{equation}

Let $C_{B,0}>0$ be the constant obtained in the pointwise
estimate of Lemma \ref{lem:projected-moser-B}, namely,
\begin{equation*}
z_n^2(x)
\geq
\frac{\log n}{2\pi}
-\frac{C_{B,0}}{4\pi^2}
+o_n(1)
\end{equation*}
uniformly for $x\in B(x_0,d/n)$. By \eqref{eq:def-CB} we have that,
\begin{equation}
C_B
=
C_{B,0}
+\frac{8\pi^2K_k^2}{\sigma}.
\label{eq:definition-CB}
\end{equation}
Combining this estimate with
\eqref{eq:un-exponential-lower-bound}, we obtain,
for $x\in B(x_0,d/n)$,
\begin{equation*}
\alpha_0u_n^2(x)
\geq
2\log n-\frac{C_B}{\pi}+o_n(1).
\end{equation*}
Therefore,
\begin{align*}
\int_{B(x_0,d/n)}
\exp(\alpha_0u_n^2)\,\mathrm{d}x
&\geq
\exp\left(
2\log n-\frac{C_B}{\pi}+o_n(1)
\right)
\vert B(x_0,d/n)\vert 
\\
&=
\pi d^2
\exp\left(
-\frac{C_B}{\pi}+o_n(1)
\right).
\end{align*}
Taking the lower limit gives
\begin{equation}
\liminf_{n\to+\infty}
\int_{B(x_0,d/n)}
\exp(\alpha_0u_n^2)\,\mathrm{d}x
\geq
\pi d^2\exp\left(-\frac{C_B}{\pi}\right).
\label{eq:un-sharp-concentration}
\end{equation}
Using the lower growth estimate on \(B(x_0,d/n)\), the positivity from
\ref{f5} outside this ball, and \eqref{eq:critical-asymptotic-lower-bound}, we obtain
\begin{equation}
\liminf_{n\to+\infty}
\int_\Omega f(x,u_n)u_n\,\mathrm{d}x
\geq
\beta \pi d^2
\exp\left(-\frac{C_B}{\pi}\right).
\label{eq:nonlinear-sharp-lower-bound}
\end{equation}
On the other hand, by
\eqref{eq:maximizer-Nehari-identity},
\begin{equation*}
\int_\Omega f(x,u_n)u_n\,\mathrm{d}x
=
\Vert u_n\Vert_B^2
=
\Vert v_n\Vert_B^2+t_n^2.
\end{equation*}
Since $v_n\to0$ in $\mathcal V_k$ and $t_n^2\to T^2$, we
obtain
\begin{equation*}
\lim_{n\to+\infty}
\int_\Omega f(x,u_n)u_n\,\mathrm{d}x
=
T^2
=
\frac{4\pi}{\alpha_0}.
\end{equation*}
Combining this identity with
\eqref{eq:nonlinear-sharp-lower-bound}, we obtain
\begin{equation*}
\frac{4\pi}{\alpha_0}
\geq
\beta\pi d^2
\exp\left(-\frac{C_B}{\pi}\right).
\end{equation*}
Consequently,
\begin{equation*}
\beta
\leq
\frac{4\exp({C_B/\pi})}{\alpha_0d^2},
\end{equation*}
which contradicts the choice of $\beta$ in \eqref{eq:choice-beta}. Therefore, there exists
$n\in\mathbb N$ such that
\begin{equation*}
\max_{u\in Q_n}\mathcal{I}(u)<\frac{2\pi}{\alpha_0}.
\end{equation*}

Finally, since the identity map belongs to $\Gamma_n$, we have
\begin{equation*}
\hat{c}_n\leq\max_{u\in Q_n}\mathcal{I}(u),
\end{equation*}
which together with $\hat{c}_n\geq a>0$, yields
\begin{equation*}
0<\hat{c}_n<\frac{2\pi}{\alpha_0}.
\end{equation*}
This completes the proof.
\end{proof}

\subsection{Palais-Smale condition in the saddle point case}
We now establish the compactness result required to conclude the
saddle point argument. The boundedness of Palais-Smale sequences and
the convergence of the nonlinear terms have already been proved in
Lemmas~\ref{eq:limitacaoC} and~\ref{eq:limitesu0C}. Therefore, we only
record the growth estimate provided by the local assumption
\ref{f9}, prove the nontriviality of the weak limit, and verify the
strong convergence.

\begin{lemma}\label{lem:nontrivial-limit-saddle}
Assume that \ref{f1}, \ref{f3}, \ref{f5}, and \ref{f9}
hold. Let $(u_n)_{n\in\mathbb N}$ be a $(PS)_c$ sequence for
$\mathcal{I}$ with
\begin{equation*}
0<c<\frac{2\pi}{\alpha_0}.
\end{equation*}
If
$
u_n\rightharpoonup u_0,
$
weakly in $H_0^1(\Omega)$, then $u_0\neq0$.
\end{lemma}

\begin{proof}
Assume, by contradiction, that $u_0=0$. In view of Lemma~\ref{eq:limitesu0C}~$(ii)$ we have
\begin{equation*}
\int_\Omega F(x,u_n)\,\mathrm{d}x
\rightarrow
\int_\Omega F(x,u_0)\,\mathrm{d}x
=
0.
\end{equation*}
Since $\mathcal{I}(u_n)\to c$, it follows that
\begin{equation*}
\frac12\Vert u_n\Vert_B^2
\rightarrow c
<
\frac{2\pi}{\alpha_0}.
\end{equation*}
Consequently, we may choose $\alpha>\alpha_0$, $q>1$ close
to $1$, and $\tau<4\pi$ such that
$
q\alpha\Vert u_n\Vert_B^2<\tau<4\pi
$
for $n\in\mathbb{N}$ sufficiently large. By \ref{f1} and \ref{f9} we have that,
\begin{equation} \label{eq:asin}
\vert f(x,u_n)u_n\vert
\leq
C_1\vert u_n\vert^2
+
C_2\vert u_n\vert\exp(\alpha u_n^2).
\end{equation}
Hence, by H\"older's inequality and
Lemma~\ref{eq:TMeditadaC},
\begin{equation*}
\left\vert
\int_\Omega f(x,u_n)u_n\,\mathrm{d}x
\right\vert
\leq
C_1\Vert u_n\Vert_{L^2(\Omega)}^2
+
C_3\Vert u_n\Vert_{L^{q'}(\Omega)},
\end{equation*}
where $q'=q/(q-1)$. In view of Lemma~\ref{eq:limitacaoC} and the compact Sobolev
embeddings, we have that
$
u_n\rightarrow 0,
$
strongly in $L^p(\Omega)$, for every $1\leq p<+\infty$. Therefore,
\begin{equation*}
\int_\Omega f(x,u_n)u_n\,\mathrm{d}x
\rightarrow 0, \quad \mbox{as } n\rightarrow\infty.
\end{equation*}
Since $(u_n)$ is bounded and $\mathcal{I}'(u_n)\to0$, we have
$
\mathcal{I}'(u_n)u_n=o_n(1).
$
Thus,
\begin{equation*}
\Vert u_n\Vert_B^2
=
\int_\Omega f(x,u_n)u_n\,\mathrm{d}x+o_n(1)
\rightarrow0.
\end{equation*}
Hence
\begin{equation*}
\mathcal{I}(u_n)
=
\frac12\Vert u_n\Vert_B^2
-
\int_\Omega F(x,u_n)\,\mathrm{d}x
\rightarrow0,
\end{equation*}
which contradicts $\mathcal{I}(u_n)\to c>0$. Therefore $u_0\neq0$.
\end{proof}

\begin{proposition}\label{prop:ps-saddle}
Assume that \ref{f1}, \ref{f3}, \ref{f5}, and \ref{f9}
hold. Then $\mathcal{I}$ satisfies the $(PS)_c$ condition for every
\begin{equation*}
0<c<\frac{2\pi}{\alpha_0}.
\end{equation*}
\end{proposition}

\begin{proof}
Let $(u_n)_{n\in\mathbb N}\subset H_0^1(\Omega)$ be a
$(PS)_c$ sequence. By Lemma~\ref{eq:limitacaoC}, $(u_n)_{n \in \mathbb{N}}$ is
bounded in $H_0^1(\Omega)$. Hence, up to a subsequence, there
exist $u_0\in H_0^1(\Omega)$ and $\rho_0\geq0$ such that
$
u_n\rightharpoonup u_0
$
weakly in $H_0^1(\Omega)$
and
$
\Vert u_n\Vert_B\rightarrow\rho_0,
$
as $n\rightarrow\infty$. Moreover, by Lemma~\ref{eq:limitesu0C}, we have
\begin{equation*}
f(x,u_n)\rightarrow f(x,u_0),
\quad\text{strongly in }L^1(\Omega)
\end{equation*}
and
\begin{equation*}
\int_\Omega F(x,u_n)\,\mathrm{d}x
\rightarrow
\int_\Omega F(x,u_0)\,\mathrm{d}x, \quad \mbox{as } n\rightarrow\infty.
\end{equation*}
It follows from Lemma~\ref{lem:nontrivial-limit-saddle} that $u_0\neq0$. We first prove that $\mathcal{I}'(u_0)=0$. Let
$\varphi\in C_c^\infty(\Omega)$. Since $\mathcal{I}'(u_n)\to0$, we
have
\begin{equation*}
\langle u_n,\varphi\rangle_B
-
\int_\Omega f(x,u_n)\varphi\,\mathrm{d}x
=
o_n(1).
\end{equation*}
Passing to the limit and using
Lemma~\ref{eq:limitesu0C}~$(i)$, we obtain
\begin{equation*}
\langle u_0,\varphi\rangle_B
=
\int_\Omega f(x,u_0)\varphi\,\mathrm{d}x.
\end{equation*}
Since $\mathcal{I}\in C^1(H_0^1(\Omega),\mathbb{R})$, we have
$\mathcal{I}'(u_0)\in H^{-1}(\Omega)$. Hence, by the density of
$C_c^\infty(\Omega)$ in $H_0^1(\Omega)$, it follows that
$\mathcal{I}'(u_0)=0$ in $H^{-1}(\Omega)$.
It remains to prove that $u_n\to u_0$ strongly in
$H_0^1(\Omega)$. Since $\mathcal{I}'(u_0)=0$, assumption \ref{f5}
gives
\begin{align*}
\mathcal{I}(u_0)
&=
\mathcal{I}(u_0)-\frac12\mathcal{I}'(u_0)u_0
\\
&=
\int_\Omega
\left[
\frac12f(x,u_0)u_0-F(x,u_0)
\right]\mathrm{d}x
\geq0.
\end{align*}
Furthermore,
\begin{equation*}
c
=
\frac12\rho_0^2
-
\int_\Omega F(x,u_0)\,\mathrm{d}x,
\end{equation*}
and therefore
\begin{equation*}
\rho_0^2-\Vert u_0\Vert_B^2
=
2\bigl(c-\mathcal{I}(u_0)\bigr)
\leq2c
<
\frac{4\pi}{\alpha_0}.
\end{equation*}
At this point, the argument used in the proof of
Proposition~\ref{lem:compactness_critical_level} can be repeated,
with \eqref{eq:asin} replacing \eqref{eq:TH3C}. Consequently,
\begin{equation*}
\int_\Omega
f(x,u_n)(u_n-u_0)\,\mathrm{d}x
\rightarrow 0,
\quad \mbox{as } n\rightarrow\infty.
\end{equation*}
As $\mathcal{I}'(u_n)\to0$ and $(u_n-u_0)$ is bounded in
$H_0^1(\Omega)$, we have
\begin{equation*}
\mathcal{I}'(u_n)(u_n-u_0)=o_n(1).
\end{equation*}
Hence,
\begin{equation*}
\langle u_n,u_n-u_0\rangle_B=o_n(1).
\end{equation*}
Using $u_n\rightharpoonup u_0$, we conclude that
\[
\Vert u_n\Vert_B^2
\rightarrow
\Vert u_0\Vert_B^2,
\quad \mbox{as } n\rightarrow\infty.
\]
Thus,
\[
u_n\rightarrow u_0,
\quad \mbox{strongly in }H_0^1(\Omega).
\]
Therefore, $\mathcal{I}$ satisfies the $(PS)_c$ condition.
\end{proof}

\subsection{Proof of Theorem~\ref{eq:teoremadoissaddle}}

Fix $n\in\mathbb N$ as provided by
Lemma~\ref{eq:lemadonivelsaddle}. By the linking construction,
the functional $\mathcal{I}$ possesses a $(PS)_{\hat{c}_n}$ sequence, where
$
0<\hat{c}_n<2\pi/\alpha_0.
$
In view of Proposition~\ref{prop:ps-saddle} we obtain a sequence such that 
$
u_j\rightarrow u
$
strongly in $H_0^1(\Omega)$, for some $u\in H_0^1(\Omega)$. Consequently,
$
\mathcal{I}'(u)=0
$
and
$
\mathcal{I}(u)=\hat{c}_n>0.
$
In particular, $u\neq0$, and hence $u$ is a nontrivial weak
solution of \eqref{nonlocalproblem}. Finally, define
$
v=L[u],
$
or equivalently,
$
\delta v=B[u].
$
By the resolvent correspondence established in Section \ref{sec2C} the pair
$(u,v)$ is a nontrivial weak solution of system \eqref{SC}.

\begin{remark}\label{rem:two-solutions}
(On the nature of the two solutions.) Theorems~\ref{eq:mudouagoraC} and~\ref{eq:teoremadoissaddle} do not provide
two solutions for the same nonlinearity. They concern disjoint classes of
nonlinearities and are proved by means of different variational arguments.
In the ground state case, condition~\ref{f2} implies that
\[
F(x,s)=o(s^{2}) \quad \text{as } s\to 0,
\]
so that the quadratic part of $\mathcal{I}$ dominates near the origin.
Consequently, $u=0$ is a strict local minimum of $\mathcal{I}$. The solution
is then obtained through the mountain pass geometry and is characterized as
a least-energy critical point, namely,
\[
\mathcal{I}(u_{0})=c_{S}.
\]

In the saddle point case, this local behavior cannot occur. Indeed,
assumption~\ref{f8} requires
\[
F(x,s)\geq \frac{1}{2}(\hat\lambda_{k}+\sigma)s^{2}
\]
for every $s\in\mathbb{R}$, where $\hat\lambda_{k}+\sigma>0$, and is therefore
incompatible with condition~\ref{f2}. In this case, the origin is a saddle
point. More precisely, assumptions~\ref{f7}-\ref{f8} imply that
$\mathcal{I}$ is positive on a sufficiently small sphere of $W_{k}$ and
negative on $V_{k}\setminus\{0\}$, where
\[
V_{k}=\operatorname{span}\{\varphi_{1},\dots,\varphi_{k}\}.
\]
Thus, $V_{k}$ provides a $k$-dimensional family of descending directions
from the origin. The solution is obtained at a linking level
$\hat c_{n}>0$, associated with a minimax construction in the
$(\dim V_{k}+1)$-dimensional space generated by $V_{k}$ and the projected
Moser function. In contrast with the solution given by
Theorem~\ref{eq:mudouagoraC}, the argument does not yield a least-energy
characterization.

The monotonicity condition~\ref{f6}, which is essential for the comparison
between the mountain pass and Nehari levels in
Theorem~\ref{eq:mudouagoraC}, is not used in
Theorem~\ref{eq:teoremadoissaddle}. In both cases, a critical point of the
reduced problem~\eqref{nonlocalproblem} gives rise to a solution of the
FitzHugh-Nagumo system~\eqref{SC} through the resolvent relation
$\delta v=B[u]$. Nevertheless, the two theorems apply to different
assumptions on the reaction term $f$ and rely on different variational
constructions. 
\end{remark}
\vspace{0.5cm}

\noindent \textbf{Acknowledgments.}  J. M. do \'O acknowledges partial support from Conselho Nacional de Desenvolvimento Científico e Tecnológico (CNPq) through grants 312340/2021-4, 409764/2023-0, 443594/2023-6, 312435/2026-6, CAPES MATH AMSUD grant 88887.878894/2023-00 and Para\'iba State Research Foundation (FAPESQ), grant no. 3034/2021. J. C. de Albuquerque acknowledges partial support from CNPq through grant 301849/2025-0.  H. H. C. e Silva acknowledges support from Para\'iba State Research Foundation (FAPESQ) through a postdoctoral fellowship, grant no. 740/2026.




\bibliography{references.bib}
\bibliographystyle{abbrv}

\end{document}